\documentclass [11pt, reqno]{amsart}
\usepackage{amsmath}
\usepackage{mathrsfs,amsmath}
\usepackage{setspace}
\usepackage{CJK}
\usepackage{mathtools}
\usepackage{amsthm}
\usepackage{amssymb}
\usepackage{enumerate}
\usepackage{verbatim}
\usepackage{bbm}
\usepackage{mathabx}

\usepackage[
hypertexnames=false, colorlinks, citecolor=black, linkcolor=blue, urlcolor=red]{hyperref}
\hypersetup{bookmarksdepth=3}

\newenvironment{entry}
{\begin{list}{X}%
		{%
			\setlength{\labelwidth}{35pt}%
			\setlength{\leftmargin}{\labelwidth}
			\addtolength{\leftmargin}{\labelsep}%
			\setlength{\itemsep}{.4pc}%
		}%
	}%
	{\end{list}}  

\normalbaselines						  %

\newcommand{\D}{\mathbb{D}}
\newcommand{\T}{\mathbb{T}}

\newcommand{\Z}{{\mathbb  Z}}
\newcommand{\N}{{\mathbb  N}}
\newcommand{\C}{{\mathbb  C}}
\newcommand{\R}{{\mathbb  R}}
\newcommand{\cH}{\mathcal{H}}

\newcommand{\cU}{\mathcal{U}}
\DeclareMathOperator\Arg{Arg}

\newcommand{\fJ}{\mathfrak{J}}

\numberwithin{equation}{section}

\newtheorem{thm}{Theorem}[section]
\newtheorem{lm}[thm]{Lemma}

\newtheorem{prop}[thm]{Proposition}
\newtheorem*{prop*}{Proposition}

\newcommand{\Ker}{\operatorname{Ker}}
\newcommand{\Ran}{\operatorname{Ran}}

\renewcommand{\ker}{\operatorname{Ker}}
\newcommand{\var}{\operatorname{var}}
\newcommand{\card}{\operatorname{card}}

\newcommand{\wt}{\widetilde}
\newcommand{\wh}{\widehat}

\newcommand{\im}{\operatorname{Im}}
\newcommand{\Dom}{\operatorname{Dom}}
\newcommand{\re}{\operatorname{Re}}
\newcommand{\ran}{\operatorname{Ran}}
\newcommand{\spn}{\operatorname{span}}
\newcommand{\cspn}{\overline{\spn}}
\newcommand{\supp}{\operatorname{supp}}
\newcommand{\rank}{\operatorname{rank}}
\newcommand{\dist}{\operatorname{dist}}

\newcommand{\dom}{\operatorname{Dom}}

\newcommand{\cC}{\mathcal{C}}
\newcommand{\cK}{\mathcal{K}}
\newcommand{\cM}{\mathcal{M}}
\newcommand{\cT}{\mathcal{T}}
\newcommand{\cE}{\mathcal{E}}

\newcommand{\cV}{\mathcal{V}}

\newcommand{\fC}{\mathfrak{C}}

\newcommand{\fc}{\mathfrak{c}}

\newcommand{\1}{\mathbf{1}}

\newcommand{\dd}{\mathrm{d}}

\newcommand{\vf}{\varphi}
\newcommand{\bvf}{\boldsymbol{\varphi}}
\newcommand{\bpsi}{\boldsymbol{\psi}}

\newcommand{\bu}{\mathbf{u}}

\newcommand{\la}{\langle}
\newcommand{\ra}{\rangle}

\newcommand{\ci}[1]{{\rule[-0.7ex]{0ex}{1.0ex}}_{\!#1}}

\newcommand{\tp}[2]{\texorpdfstring{#1}{#2}}
\newcommand{\tup}[1]{\textup#1}

\theoremstyle{remark}
\newtheorem{rem}[thm]{Remark}
\newtheorem*{rem*}{Remark}
\newtheorem{defin}[thm]{Definition}

\renewcommand{\labelenumi}{\textup{(\roman{enumi})}}

\newcounter{vremennyj}

\newcommand\cond[1]{\setcounter{vremennyj}{\theenumi}\setcounter{enumi}{#1}\labelenumi\setcounter{enumi}{\thevremennyj}}

\begin{document}
\title[Inverse spectral problem]{A dynamical system approach 
	to the inverse spectral problem for Hankel operators: the general case}
 
\author{Zhehui Liang}
\author{Sergei Treil}
\thanks{Work of S.~Treil is 
	supported  in part by the National Science Foundation under the grants  
	DMS-1856719, DMS-2154321}
\address{Department of Mathematics \\ Brown University \\ Providence, RI 02912 \\ USA}
\email{treil@math.brown.edu}
\begin{abstract}
We study the inverse problem for the Hankel operators in the general case. Following the work of 
G\'erard--Grellier, the spectral data is obtained from the pair of Hankel operators $\Gamma$ and 
$\Gamma S$, where $S$ is the shift operator. 

The theory of complex symmetric operators provides a convenient language for the description of the 
spectral data. We introduce the abstract spectral data for the general case, and  use the dynamical 
system approach, originated in the \cite{Peller 1995}, to reduce the problem to asymptotic 
stability of some contraction, constructed from the spectral data. 

The asymptotic stability is usually the hard part of the problem, but in the investigated earlier 
by G\'erard--Grellier  case of compact operators we get it almost for free. 

For the case of compact operators we get a concrete representation of the abstract spectral data as 
two intertwining sequences of singular values, and two sequences of finitely supported probability 
measures. This representation is different from one treated by G\'erard--Grellier, and we provide 
the translation from one language to the other; theory of Clark measures is instrumental there.   

\end{abstract}

\maketitle

\setcounter{section}{-1}
\section{Notation}
All operators act on or between Hilbert spaces, and we consider only separable  Hilbert spaces. 
\begin{entry}
\item[$A^*$] adjoint of the operator $A$;
\item[$|A|$] modulus of the operator $A$, $|A|:=(A^*A)^{1/2}$;
\item [$P\ci E$] the orthogonal projection  onto a subspace $E$.
\end{entry}

In this paper we will use the \emph{linear algebra notation}, identifying vector $a$ in a Hilbert 
space 
$\cH$ with the operator  $\alpha\mapsto \alpha a$ acting from scalars to $\cH$. Then the 
symbol $a^*$ denotes the (bounded) linear functional $x\mapsto ( x, a ) $. 

\section{Introduction and some preparation work}
\subsection{Introduction to Hankel operators}

A Hankel operator is a bounded linear operator in $\ell^2=\ell^2(\Z_+)$ with matrix whose entries 
depend on the sum of indices, 
\begin{align*}
\Gamma = \left(\gamma_{j+k}\right)_{j,k\ge0} \,.
\end{align*}
Denoting by $S$  the \emph{shift operator} in $\ell^2$,
\begin{align*}
S(x_0, x_1, x_2, \ldots ) &= (0, x_0, x_1, x_2, \ldots) \,, 
\intertext{and by $S^*$ its adjoint (the \emph{backward shift}), }
S^* (x_0, x_1, x_2, \ldots ) &= (x_1, x_2, x_3,  \ldots ), 
\end{align*}
we can see that an operator $\Gamma$ on $\ell^2$ is a Hankel operator if and only if 
\begin{align*}
\Gamma S = S^* \Gamma;  
\end{align*}
sometimes this formula is used as the definition of a Hankel operator. 

The Fourier Transform  $\{a_k\}\ci{k\in\Z}\mapsto \sum_{k\in\Z}a_k z^k$ on $\ell^2(\Z)$ (sometimes 
called  $z$-transform) identifies $\ell^2(\Z)$ with $L^2(\T)$, and $\ell^2(\Z_+)$ with the Hardy 
space $H^2\subset L^2(\T)$, so Hankel operators are often treated as the operators in the Hardy 
space. 

Hankel operators play important role in analysis, connecting function theory and operator theory. 
The inverse spectral problem for Hankel operators was initially motivated by the theory of stationary random 
processes.  Investigating geometry of the ``past'' and ``future'' of such processes,  V.~V.~Peller 
and S.~V.~Khruschev posed a problem of  describing all non-negative self-adjoint operators 
unitarily equivalent to the modulus of a Hankel operator.  After some preliminary results, see 
\cite{khrushchev 1984}, the problem was fully solved by S.~Treil \cite{Treil 1985}. It turned out 
that any 
non-negative 
self-adjoint operator $A$ that is not invertible and whose kernel is either infinite-dimensional or 
trivial is unitarily equivalent to the modulus of a Hankel operator (note that trivially the 
modulus of a Hankel operator is always not invertible and cannot have finite-dimensional kernel). 

Later, motivated by dynamical systems, A.~Megretskii, V.~V.~Peller, and S.~Treil completely 
described self-adjoint operators unitarily equivalent to some Hankel operator, see \cite{Peller 1995}. The 
description was more complicated than for a modulus, and involved some ``almost symmetry'' of the 
spectral measure. 

We should mention here, that in both problems the solution (a Hankel operator) is trivially not 
unique. 

Recently the interest in the inverse spectral problem for Hankel operators was renewed in  
connection with the so-called cubic Szeg\"{o} equation, which is a model completely integrable 
Hamiltonian system. In pioneering series of papers \cite{Gerard 2010}, \cite{Gerard 2012} P.~Gerard and S.~Grellier  
investigated the inverse spectral problem for Hankel operators and its connection with the Szeg\"{o} 
equation. For the case of compact operators they have the inverse problem completely solved. One of 
their discoveries was that the spectral data not for $\Gamma$, but for \emph{both} $\Gamma$ and 
$\Gamma S$ completely determines the operator $\Gamma$, so the inverse problem have a unique 
solution.

\section{Preliminaries}
\subsection{Direct integral of Hilbert spaces and model for self-adjoint operators}
Let us recall some basic facts about direct (a.k.a.\ von Neuman) integral of Hilbert spaces, and 
the model for self-adjoint operators, see \cite{Birman 1987}. In this paper we present a simplified 
version of the direct integral, which is sufficient for our purposes. 

Let $\cE$ be a separable Hilbert space, and let $\{e_k\}_{k=1}^\infty$ be an orthonormal basis in 
$\cE$. Let $\mu$ be a finite compactly supported Borel measure on $\R$, and let $N:\supp \mu \to 
\N \cup\{\infty\}$ be the  dimension function, i.e.\ a Borel measurable function (we only need 
$N$ to be defined $\mu$-a.e.). 

Denote $\cE(t) := \cspn\{e_k: 1\le k \le N(t)\}$, and define the direct integral 
\begin{align}
\label{e: DirInt}
\cH := \int_\R \!\!\oplus \,\, \cE(t)\dd\mu(t) 
\end{align}
as the subspace of the $\cE$-valued $L^2$-space $L^2(\mu; \cE)$, consisting of all functions $f\in 
L^2(\mu;\cE)$ such that 
\begin{align*}
f(t)\in \cE(t) \qquad \mu\text{-a.e.}
\end{align*}
The direct integral form of the spectral theorem states that a self-adjoint operator in a separable 
Hilbert space (recall that in this paper all Hilbert spaces are separable) is unitarily equivalent 
to the multiplication operator $M$ by the independent 
variable
\begin{align*}
Mf(t) =t f(t)
\end{align*}
in a direct integral of form \eqref{e: DirInt}. 

The spectral type $[\mu]$ of the measure $\mu$, i.e.~the equivalence class of all measures mutually 
absolutely continuous with $\mu$, and the dimension function $N$ (defined $\mu$-a.e.) give us a 
complete set of unitary invariants for a self-adjoint operator. This means that two self-adjoint 
operators are unitarily equivalent if and only if the measures in the direct integral are mutually 
absolutely continuous and the dimension functions coincide $\mu$-a.e.

\subsection{Complex Symmetric Operators}
\label{s: conj and CSO}
Matrix of a Hankel operator $\Gamma$ is symmetric, and a convenient way to explore that fact is to 
use the theory  of so called \emph{complex symmetric operators}. 

Let us recall some main  definitions and basic facts,  
cf \cite{Garcia 2007}.

\begin{defin}
An operator $\fC$ in a complex Hilbert space $\cH$ is called a \emph{conjugation}, if 
it is
\begin{enumerate}
    \item \emph{conjugate-linear}:  $\fC(\alpha x + \beta y) = \bar{\alpha}\fC x + \bar{\beta}\fC 
    y$ for all $x,y \in \cH$ ;
    \item and \emph{involutive}:  $\fC^2 = I$ ;
    \item and \emph{isometric}: $\| \fC x \| = \| x \|$ for all $x \in \cH$.
\end{enumerate}
\end{defin}

\begin{defin}
Let $\fC$ be a conjugation on Hilbert space $\cH$. A bounded linear operator $T$ on $\cH$ is called 
\emph{$\fC$-symmetric} if $T^* = \fC T \fC$.
\end{defin}

It is well-known and not hard to see that a conjugate-linear operator $\fC$ is a conjugation if and 
only if there exists an orthonormal basis $\{ e_n \}_{n=1}^{\infty}$ of $\cH$ such that  $\fC e_n = 
e_n$ for all $n$;  such basis 
is often called a \emph{$\fC$-real orthonormal basis}. 
Note that given a cojugation $\fC$ the choice of  $\fC$-real basis is not unique. 

It is easy to see that for a $\fC$-symmetric operator $T$ its matrix in any $\fC$-real orthonormal 
basis is symmetric. 

On the other hand, if a matrix of an operator $T$ in some orthonormal basis $\{e_n\}_{n}$ is 
symmetric, then $T$ is $\fC$-symmetric for the conjugation $\fC$ given by 
$\fC(\sum\limits_{n}\alpha_n e_n) = \sum\limits_{n}\overline{\alpha}_n e_n$. Note that a 
conjugation $\fC$ that $T$ is $\fC$-symmetric is generally not unique, and the above is just one 
possible choice.  

The matrix of a Hankel operator in the standard basis in $\ell^2$ is symmetric, thus Hankel 
operator is $\fC$-symmetric with respect to the canonical conjugation on $\ell^2$, 
\begin{align}
\label{e: canonical conjugation}
    \fC(z_1,z_2,z_3, \ldots ) = (\bar{z}_1,\bar{z}_2,\bar{z}_3 , \ldots) . 
\end{align}

Finally, we will need a simple but important formula for an arbitrary conjugation $\fC$: 
\begin{align}
    \label{property of conjugation for an inner product}
    \langle \fC x , y \rangle = \langle \fC y , x \rangle\qquad \forall x, y\in\cH. 
\end{align}
It can be easily proved using properties \cond1--\cond3 of a conjugation, or just by decomposing 
$x$ and $y$ in a $\fC$-real orthonormal basis.

\subsection{Polar decomposition of \texorpdfstring{$\fC$}{C}-symmetric operators}

Let us recall that an operator $U$ is called a \emph{partial isometry} if its restriction to $(\ker 
U)^\perp$ is an isometry; note that $(\ker U)^\perp$ does not need to be $U$-invariant.

A bounded operator in a Hilbert space admits a  \emph{polar decomposition} 
$T = U |T|$, where $|T|:= (T^*T)^{1/2}$, and $U$ is a partial isometry with $\ker U=\ker T$; note 
that under these assumptions the partial isometry $U$ is unique.  Sometimes the partial isometry 
$U$ can be replaced by a unitary operator (even in $\ker T \ne\{0\}$), although unlike the 
finite-dimensional case this is not always possible. 

For a complex symmetric operator  we can say a bit more about its polar decomposition. Following  
\cite{Garcia 2007}, we say that    
a conjugate-linear operator $\fJ$ in Hilbert space is  a \emph{partial conjugation} if $(\ker 
\fJ)^\perp$ is invariant for $\fJ$ and $\fJ |\ci{(\Ker \fJ)^{\perp}}$ is a conjugation. 
Again, following \cite{Garcia 2007} we say a conjugation $\fJ$ is supported on (a subspace) $ \cK$ 
if $\cK = (\ker \fJ)^\perp$.  

We will need the following theorem,  \cite[Theorem 2]{Garcia 2007}:

\begin{thm}
\label{t: GPD}
A bounded  $\fC-$symmetric operator $T$ admits the polar decomposition $T = U |T|$ with $U=\fC\fJ$ 
where  $\fJ$ is a partial conjugation, $\ker \fJ =\ker T$,  commuting with 
$|T|$.

\end{thm}

\begin{rem}
Note that in the above theorem the partial conjugation $\fJ$ is unique and the partial isometry $U$ 
is $\fC-$symmetric. 
\end{rem}
\begin{rem}
\label{r: partial conjugation can be extended to a conjugation}
We can always find a conjugation $\wt \fJ$, such that $\wt \fJ = \fJ$ on $(\Ker 
\fJ)^{\perp}$. To do that we can take an arbitrary conjugation $\fJ_1$ on $\Ker 
\fJ$, and then define $\wt \fJ\big|_{(\ker\fJ)^\perp} := \fJ_1$. 
Then $T = \fC\wt\fJ |T| = \wt U |T|$, where $\wt U =\fC\wt\fJ$ is clearly a $\fC$-symmetric unitary 
operator. 
\end{rem}

\subsection{Abstract spectral data for Hankel operators}
\label{s: ASD Hankel}

For a Hankel operator $\Gamma$ we will always denote $\Gamma_1:= \Gamma S = S^*\Gamma$ (the 
identity $\Gamma S = S^*\Gamma$ is an alternative definition of a Hankel operator, see the 
Introduction above). 

W can write 
\begin{align*}
    |\Gamma|^2 - |\Gamma_1|^2 = \Gamma^* \Gamma - \Gamma^* S S^* \Gamma = \Gamma^* (I - SS^*) 
    \Gamma = \Gamma^* (e_0 e_0^*) \Gamma, 
\end{align*}
so denoting $u := \Gamma^* e_0$ we have the rank one perturbation relation
\begin{align}
    \label{equation for Hankel}
    |\Gamma|^2 - |\Gamma_1|^2 = uu^* . 
\end{align}

We will need the following simple statement. 
\begin{lm}
\label{l:Jp}
Let  $R=R^*$ be  a self-adjoint operator in a Hilbert space $\cH$, and let $p\in\cH$. 

There exists a conjugation $\fJ_p$ commuting with  $R$  and preserving $p$, i.e.~such that
$\fJ_p p=p$. 

Moreover, $\fJ_p$ is uniquely determined on
$\cH_0:=\cspn\{ R^n p: n\ge0\}$, and $\cH_0$ is a reducing subspace for $\fJ_p$, meaning that both 
$\cH_0$ and $\cH_0^\perp$ are $\fJ_p$-invariant. 

Finally,  $\fJ_p$ is unique if and only if $p$ is cyclic for $R$.
\end{lm}
\begin{proof}
Defining $\fJ_p$ on $\cH_0$ by  
\begin{align}
\label{e: fJ_p 01}
\fJ_p\Bigl(\sum_{k\ge0} \alpha_k R^k p \Bigr) = \sum_{k\ge0} \overline\alpha_k R^k p , 
\end{align}
we get a conjugation on $\cH_0$, commuting with $R\big|_{\cH_0}$. 

To extend $\fJ_p$ to all $\cH$ we just need to construct a conjugation on $\cH_0^\perp$ commuting 
with $R\big|_{\cH_0^\perp}$, or equivalently, construct a conjugation commuting with the 
multiplication operator $M$ by the independent variable in the direct integral representation 
\eqref{e: DirInt} of 
$R\big|_{\cH_0^\perp}$.  But that is trivial, the conjugation 
\begin{align*}
\sum_k f_k e_k \mapsto \sum_k \overline f_k e_k
\end{align*}
($f_k$ are scalar-valued functions) give one possibility. 

As for the uniqueness, if $\fJ_p$ commutes with $R$ then $\fJ_p R^k p = R^k p$ for all $k\ge 0$, so 
by conjugate linearity $\fJ_p$ on $\cH_0$ must be given by \eqref{e: fJ_p 01}. We can see from 
this formula that $\fJ_p\cH_0=\cH_0$. Since $\fJ_p$ is an isometry, it preserves orthogonality, so 
$\fJ_p (\cH_0^\perp)\subset \cH_0^\perp$. 

Finally, a conjugation commuting with the multiplication $M$ by the independent variable is clearly 
not unique, so $\fJ_p$ is unique if and only if $\cH_0=\cH$. 
\end{proof}

The lemma below gives a complete description of all conjugations $\fJ_p$ from Lemma \ref{l:Jp}. 
\begin{lm}
\label{l:equiv of Jp}
Let $\fJ_p$ be a conjugation from Lemma \ref{l:Jp}.  Then any other such conjugation $\fJ_p'$   is 
given by $\fJ_p'=\Psi \fJ_p$, where $\Psi$ is unitary $\fJ_p$-symmetric operators commuting with 
$R$ and preserving $p$, $\Psi p=p$. 
\end{lm}
\begin{proof}
If $\fJ_p'$ is another conjugation, commuting with $P$ and preserving $p$, then $\fJ_p'= = \Psi 
\fJ_p$, where $\Psi:=\fJ_p'\fJ_p$. It is easy to see that $\Psi$ is a unitary operator, commuting 
with $R$ and preserving $p$. It is also 
easy to see that $\psi$ is $\fJ_p$-symmetric. 

On the other hand, if $\Psi$ is a $\fJ_p$-symmetric unitary operator, commuting with $R$ and such 
that $\Psi p =p$, then the (conjugate-linear) operator $\fJ'_p := \psi \fJ_p$ is a conjugation 
commuting with $R$ and preserving $p$. 

Indeed, the operator $\fJ_p'$ is trivially conjugate-linear, isometric, preserves $p$ and commutes 
with $R$. To show that $\fJ_p'$ is a conjugation, it remains to show that $\fJ_p'$ is an 
involution.  Using the $\fJ_p$-symmetry of 
$\Psi$, we get that $\Psi\fJ_p=\fJ_p\Psi^*$, and so
\[
(\fJ_p')^2 = \Psi \fJ_p \Psi \fJ_p =  \fJ_p \Psi^* \Psi \fJ_p = \fJ_p^2 =I.
\]
\end{proof}

\begin{rem}
Since $\Psi$ commutes with $R$ and preserves $p$, it is easy to see that $\Psi\big|_{\cH_0} = 
I|\ci{\cH_0}$.  
\end{rem}

\section{Spectral data for Hankel operators}
\label{s:SD Hankel}
Returning to Hankel operators, we apply Theorem  \ref{t: GPD} to the operators $\Gamma$ and 
$\Gamma_1 =\Gamma S =S\Gamma$ to get the polar decomposition
\begin{align}
\label{e:PolDec01}
\Gamma = \fC \fJ |\Gamma|, \qquad \Gamma_1 = \fC\fJ_1 |\Gamma_1| ,  
\end{align}
where $\fJ$ and $\fJ_1$ are partial conjugations, $\ker \fJ=\ker\Gamma$, $\ker \fJ_1=\ker\Gamma_1$, 
commuting with $|\Gamma|$ and $|\Gamma_1|$ respectively.  

To further investigate conjugations $\fJ$ and $\fJ_1$, let us apply Lemma \ref{l:Jp} to the 
operator $|\Gamma^*|$ and the vector $u=\Gamma^* e_0$. We get a conjugation $\fJ_u$ commuting with 
$|\Gamma|$ (and so with $|\Gamma|^2$), and preserving $u$, $\fJ_u u = u$. The rank one perturbation 
relation \eqref{equation for Hankel} then implies that $\fJ_u$ commutes with $|\Gamma_1|^2$, and 
therefore with $|\Gamma_1|$.

The polar decomposition \eqref{e:PolDec01} then can be rewritten as 
\begin{align}
\label{e:PolDec01a}
\Gamma & = \fC\Phi \fJ_u |\Gamma|, \qquad \Gamma_1 =\fC \Phi_1\fJ_u |\Gamma_1| ,
\intertext{or, equivalently}
\label{e:PolDec01b}
\fC\Gamma & = \Phi \fJ_u |\Gamma|, \qquad \fC\Gamma_1 = \Phi_1\fJ_u |\Gamma_1| ,
\end{align}
where $\Phi= \fJ \fJ_u$, $\Phi_1 = \fJ_1\fJ_u$ are $\fJ_u$-symmetric partial isometries, $\ker 
\Phi=\ker\Gamma$, $\ker \Phi_1= \ker \Gamma_1$. 

The operators $\fJ_u$, $\Phi$, $\Phi_1$ are not unique, but by Lemma \ref{l:equiv of Jp}, if we fix 
on triple $\fJ_u$, $\Phi$, $\Phi_1$, all other triples are given by  $\Psi  \fJ_u$, $\Phi \Psi^*$, 
$\Phi_1\Psi^*$, where $\Psi$ runs over all unitary operators commuting with $|\Gamma|$ and 
preserving 
$u$, $\Psi u =u$. Note that $\ker \Gamma$ is a reducing subspace for $\Psi$, and the operators 
$\fJ_u$, $\Phi$, $\Phi_1$ are zero on $\ker \Gamma$, so only action of $\Psi$ on 
$(\ker\Gamma)^\perp$ matters.

Let us restrict everything to the essential parts. Denote
\begin{align*}
\wt\Gamma &= \Gamma\big|_{(\ker\Gamma)^\perp}, \qquad 
\wt\Gamma_1 = \Gamma_1\big|_{(\ker\Gamma)^\perp}, \qquad
\wt\fJ_u = \fJ_u \big|_{(\ker\Gamma)^\perp}, \\
\wt \Phi &=\Phi\big|_{(\ker\Gamma)^\perp},\qquad
\wt \Phi_1 =\Phi_1\big|_{(\ker\Gamma)^\perp},\qquad
\wt \Psi =\Psi\big|_{(\ker\Gamma)^\perp}.
\end{align*}

One can see that 
\begin{align}
\label{e: rk1 pert 02}
|\wt\Gamma|^2 - |\wt\Gamma_1|^2 = uu^*
\end{align}
Notice that $\ker\wt\Gamma=\{0\}$, so $\fJ_u$ is a conjugation, and $\wt\Phi$ is a unitary 
operator. We should also mention that $\wt \Gamma_1$ can have a one-dimensional kernel, so we can 
only guarantee that $\wt\Phi_1$ is a partial isometry.  

So, given a Hankel operator $\Gamma$ (and therefore the operator $\Gamma_1=\Gamma S=S^*\Gamma$), we 
constructed a triple $\wt \fJ_u$, $\wt\Phi$, $\wt\Phi_1$, such that

\begin{enumerate}
\item $\wt\fJ_u$ is a conjugation commuting with $|\wt\Gamma|$ and preserving $u=\Gamma^*e_0$;
\item $\wt\Phi$ is a $\fJ_u$-symmetric unitary operator commuting with $|\wt\Gamma|$; 
\item $\wt\Phi_1$ is a $\fJ_u$-symmetric partial isometry, $\ker \wt\Phi_1=\ker\wt\Gamma_1$,  
commuting with $|\wt\Gamma_1|$.  
\end{enumerate}
The triple $\wt \fJ_u$, $\wt\Phi$, $\wt\Phi_1$ is generally not unique, but any other such triple 
$\wt \fJ_u'$, $\wt\Phi'$, $\wt\Phi_1'$ is given by 
\begin{align}
\label{e:data equiv rel 01}
\wt \fJ_u' = \wt\Psi \wt \fJ_u, \qquad \wt \Phi' = \wt\Phi \wt\Psi^*, \qquad 
\wt \Phi_1' = \wt\Phi_1 \wt\Psi^*, 
\end{align}

Finally, let us notice that $(\ker\Gamma)^\perp$ is invariant for $\fC\Gamma$ and $\fC\Gamma_1$, 
see Lemma \ref{l:C Gamma invariance} below, so we can restrict the identities \eqref{e:PolDec01b} 
to $(\ker\Gamma)^\perp$, to get 
\begin{align}
\label{e:PolDec04}
\fC\wt\Gamma & = \wt\Phi \wt\fJ_u |\wt\Gamma|, \qquad \fC\wt\Gamma_1 = \wt\Phi_1\wt\fJ_u 
|\wt\Gamma_1 | .
\end{align}

We claim that the tuple $|\wt\Gamma|$, $|\wt\Gamma_1|$, $\wt \fJ_u$, $\wt\Phi$, $\wt\Phi_1$, $u$, 
defined up to unitary equivalence and the equivalence relation \eqref{e:data equiv rel 01} 
completely defnines the Hankel operator $\Gamma$, see the details in the next section. 

\begin{lm}
	\label{l:C Gamma invariance}
	Let $\Gamma$ be a Hankel operator, and $\fC$ be the canonical conjugation on $\ell^2$ given by 
	\eqref{e: canonical conjugation}. 
	
	The subspace $(\ker\Gamma)^\perp$ is an invariant subspace for $\fC\Gamma$ and $\fC\Gamma_1$. 
\end{lm}
Note that $(\ker\Gamma)^\perp$ is generally not invariant for $\Gamma$. 

\begin{proof}[Proof of Lemma \ref{l:C Gamma invariance}]
Since $\fC \Gamma = \fJ |\Gamma|$, and $(\ker \Gamma)^\perp = (\ker|\Gamma|)^\perp$ is trivially invariant for $|\Gamma|$, to prove that $(\Ker \Gamma)^{\perp}$ is an invariant subspace for $\fC \Gamma$, it's sufficient to show that $(\Ker \Gamma)^{\perp}$ is an invariant subspace for  $\fJ$.

 Take any $ x \perp \Ker \Gamma$, and any $y \in \Ker \Gamma$. Since $\fJ$ commutes with $|\Gamma|$, we have $\fJ y \in \Ker \Gamma$. Thus 
\begin{align*}
   \langle \fJ x, y \rangle = \langle \fJ y , x \rangle = 0, 
\end{align*}
and $\fJ x \perp \Ker \Gamma$.

As for $\fC \Gamma_1$, since
\begin{align*}
   \fC \Gamma_1 \fJ_u = \Phi_1 |\Gamma_1| = S^* \Phi |\Gamma| =S^* \fC \Gamma \fJ_u
\end{align*}
and $(\Ker \Gamma)^{\perp}$ is an invariant subspace for $\fJ_u$, hence $(\Ker \Gamma)^{\perp}$ is an invariant subspace for $\fC \Gamma_1 \fJ_u$, thus also for $\fC \Gamma_1$.
\end{proof}

\section{Abstract inverse spectral problem for Hankel operators}
\label{s: Abstract ISP}
In this section, we consider the inverse spectral problem for  Hankel operators $\Gamma$ with 
abstract spectral data. 

\subsection{Setup} 
\label{s: Abstract ISP setup}
 Assume that we are given a tuple $(R, R_1, p,\fJ_p,\bvf, \bvf_1)$, where 
\begin{enumerate}
\item $R, R_1\ge 0$,  are self-adjoint operators in $\cH$, $\ker R=\{0\}$, and $p\in\cH$, such 
that 
\[
R^2 -R_1^2 = pp^*.
\]
Note that the above identity implies that $p\in \ran R$ and $\|R^{-1}p\|\le 1$. 

\item $\fJ_p$ is a conjugation commuting with $R$ and $R_1$ and preserving $p$, $\fJ_p p =p$. 

\item $\bvf$ is a $\fJ_p$-symmetric unitary operator commuting with $R$. 

\item $\bvf_1$ is a $\fJ_p$-symmetric partial isometry, $\ker \bvf_1=\ker R_1$, commuting 
with $R_1$; 
note that if $\ker R_1=\{0\}$ then $\bvf_1$ is unitary. 
\end{enumerate}

 Note, that the operator $R_1$ can have a non-trivial (one-dimensional) kernel.
\begin{prop}
\label{r:ker R_1}
 $\ker R_1\ne\{0\}$ is and only if 
\begin{align}
\label{e:kerR_1} 
p \in \ran R^2, \qquad \text{and}\qquad \|R^{-1}p\|=1 .
\end{align}
\end{prop}
\begin{proof}
Let $R_1x=0$ for $x\ne0$. Then $R_1^2 x=0$, so 
\begin{align}
\label{e:R_1^2 x = pp^*x}
R^2 x = \la x,p\ra p, 
\end{align}
which means that $p=\alpha R^{-2}x$, $\alpha\ne 1$, i.e.~that $p\in \ran R^2$.  
By homogeneity we can assume without loss of generality that $\alpha=1$, and substituting 
$x=R^{-2}p$ into \eqref{e:R_1^2 x = pp^*x} we get that 
\begin{align*}
p= \la R^{-2} p, p \ra p , 
\end{align*}
so $\|R^{-1}p\|^2 = \la R^{-2} p, p \ra =1$. 

Reversing the above reasoning, we see that conditions \eqref{e:kerR_1} imply \eqref{e:R_1^2 x = 
pp^*x} for $x=R^{-1}p$, so $x\in\ker R_1^2 = \ker R_1$. 
\end{proof}

\begin{rem}
\label{varphi 1 is partial unitary}
Note that $\bvf_1|\ci{(\Ker R_1)^{\perp}}$ is unitary, implied by \cond4. In fact, since we 
already have $\bvf_1|\ci{(\Ker R_1)^{\perp}}$ is isometry, it suffices to show that 
$\bvf_1|\ci{(\Ker R_1)^{\perp}}$ is onto. From equation $\bvf_1 \fJ_p = \fJ_p \bvf_1^*$, 
since $\Ran \bvf_1^* = (\Ker R_1)^{\perp}$ and $(\Ker R_1)^{\perp}$ is a reducing subspace for 
$\fJ_p$, we have $(\Ker R_1)^{\perp} \subseteq \Ran (\fJ_p \bvf_1^*)$, hence $(\Ker R_1)^{\perp} 
\subseteq \Ran \bvf_1$.
\end{rem}

We want to know whether we can find a Hankel operator $\Gamma$ and choose a appropriate  
conjugation $\wt\fJ_u$ from the equivalence class defined by \eqref{e:data equiv rel 01}, 
such that the tuple $(|\wt \Gamma|,|\wt \Gamma_1|, \wt \fJ_u , \wt \Phi, \wt \Phi_1, u)$ defined in 
equation \eqref{e:PolDec04} is  unitary equivalent to $(R,R_1, \fJ_p, \bvf, 
\bvf_1, p)$, i.e.that
\begin{align}
\label{e:UE01}
|\wt \Gamma| & = \wt \cV R \wt \cV^*, \qquad |\wt \Gamma_1| = \wt \cV R_1 \wt \cV^*; \\
\label{e:UE02}
    \wt \Phi = \wt \cV \bvf \wt \cV^*, \qquad \wt \Phi_1 &= \wt \cV \bvf_1 \wt \cV^*, \qquad 
    \wt \fJ_u = \wt \cV \fJ_p \wt \cV^*, \qquad u=\wt \cV p, 
\end{align}
for some unitary operator $\wt \cV:\cH \to (\ker \Gamma)^\perp$.

We can present a simple necessary condition. Since the canonical conjugation $\fC$ in $\ell^2$, see 
\eqref{e: canonical conjugation} commutes with the backward shift $S^*$, we can rewrite the 
identity $\Gamma_1 = S^*\Gamma$ as 
\begin{align*}
\fC\Gamma_1 &= S^* \fC\Gamma, 
\intertext{and since $(\ker\Gamma)^\perp$ is invarianf for $\fC\Gamma$ and $\fC\Gamma_1$, see Lemma 
\ref{l:C Gamma invariance}, we can 
restrict this identity  to $(\ker\Gamma)^\perp$,}
\fC\wt\Gamma_1 &= S^*\big|_{(\ker\Gamma)^\perp} \fC\wt\Gamma. 
\end{align*}
Using \eqref{e:PolDec04} we can rewrite it as
\begin{align*}
\wt\Phi_1 |\wt\Gamma_1| \wt \fJ_u & = S^*\big|_{(\ker\Gamma)^\perp} \wt\Phi |\wt\Gamma| \wt \fJ_u 
\intertext{or equivalently, factoring out $\wt\fJ_u$,}
\wt\Phi_1 |\wt\Gamma_1|  & = S^*\big|_{(\ker\Gamma)^\perp} \wt\Phi |\wt\Gamma| . 
\end{align*}
Define $\Sigma^*:=\wt\cV S^*\big|_{(\ker\Gamma)^\perp}\wt\cV^*$. We can see from the unitary 
equivalence \eqref{e:UE01}, \eqref{e:UE02} that
\begin{align}
\notag
\bvf_1 R_1 &= \Sigma^* \bvf R, 
\intertext{so}
\label{e:Sigma*01}
\Sigma^* &= \bvf_1 R_1 R^{-1}\bvf^* = R_1\bvf_1 \bvf^* R^{-1} . 
\end{align}

The  justification of the above formula is given by the following well known simple lemma, which we 
present without proof, applied to $A= R$, $B=R_1$ (clearly $R_1^2\le R^2$). This trivially 
justifies the first part of the formula. As for the second part, since $\bvf^*$ is an invertible 
operator, commuting with $R$, we know that $\bvf \ran R= \ran R$, and since $\dom R^{-1}=\ran R$ we 
conclude that $R^{-1}\bvf^* = \bvf^* R^{-1}$ (in the strongest sense, with the equality of 
domains), see Remark \ref{r:commutation} below.

\begin{lm}[Douglas Lemma]
\label{l:Douglas}
Let $A$ and $B$ be bounded operators in a Hilbert space $\cH$ such that 
\[
\|Bx \|\le \| Ax \| \qquad \forall x \in\cH.
\]
\tup{(}or, equivalently, $B^*B\le A^*A$\tup{)}, and let $\ker A=\ker A^* =\{0\}$. 

Then the operator $T=BA^{-1}$ \tup{(}defined initially on a dense set\tup{)} extends to a 
contraction, 
$\|T\|\le 1$ and its adjoint is given by $T^* = (A^*)^{-1}B^*$; notice that the condition $\|T\|\le 
1$ implies that $\ran B^*\in \Dom (A^*)^{-1}$, so the operator $T^*$ is well defined on all 
$x\in\cH$.  
\end{lm}

\begin{rem}
\label{r:commutation}
If $A$is a bounded operator with $\ker A=\ker A^* =\{0\}$, and $B$ is an invertible operator, 
commuting with $A$, then trivially $B \ran A = B^{-1} \ran A = \ran A$. 

If  $A^{-1}$ denotes the \emph{natural} inverse of $A$, i.e.~the densely defined operator with 
$\dom A^{-1}  = \ran A$, then 
\begin{align*}
BA^{-1}=A^{-1}B
\end{align*}
 with the natural domains of both $BA^{-1}$ and 
$A^{-1}B$ being $\ran A$. 

In this paper we will often use this commutation relation.
\end{rem}

Recall the following simple definition. 
\begin{defin}
\label{d: asymptoic stability}
An operator $T$ is called \emph{asymptotically stable} if  $ T^n\to 0$  in the strong 
operator 
topology as $n\to \infty$, i.e. if    
\begin{align*}
\lim\limits_{n \to \infty}\|T^{n}x\| \rightarrow 0, \qquad \forall x\in\cH.
\end{align*}
\end{defin}

Since the backward shift $S^*$ is trivially asymptotically stable, so is its restriction to any 
$S^*$-invariant subspace. In particular,   $S^*\big|_{(\ker\Gamma)^\perp}$ is asymptotically 
stable,  so if we have unitary equivalence \eqref{e:UE01}, \eqref{e:UE02}, the operator 
$\Sigma^* = \bvf_1 R_1 R^{-1}\bvf^*$ must be asymptotically stable. 

Turns out that this condition is also sufficient. 
\subsection{Results}
\label{s:AbstrISP-results}
\begin{thm}
\label{t:AbstrISP}
If contraction  $\Sigma^* := \bvf_1 R_1 R^{-1}\bvf^*$  is 
asymptotically stable, then there exists a unique 
Hankel operator $\Gamma$ such that
\begin{align}
\label{e:HankUnitaryRepr 01}
\fC  \Gamma &=  \cV R \bvf \fJ_p  \cV^*, \\
\label{e:HankUnitaryRepr 02}
\fC  \Gamma_1 &=  \cV R_1 \bvf_1 \fJ_p  \cV^*, \\
\label{e:u UnitaryRepr 01}
\Gamma^* e_0 &=  \cV p 
\end{align}
for some isometry $\cV: \cH \to \ell^2$. 
Furthermore,  $\Ker \Gamma = \{0\}$ if and only if $\| R^{-1}p \| = 1$ and $R^{-1} \notin \Ran R$ 
\textup(recall that $\|R^{-1} p\|\le 1$\textup).
\end{thm}

Since $\ker R=\{0\}$, we can conclude that $\ran\cV=(\ker\Gamma)^\perp$ (so $\cV$ is unitary if 
$\ker\Gamma=\{0\}$). If we denote by $\wt\cV$ the operator $\cV$ with the target space restricted 
to $\ran \cV = (\ker\Gamma)^\perp$, we can rewrite identities \eqref{e:HankUnitaryRepr 
01}, \eqref{e:HankUnitaryRepr 02}, \eqref{e:u UnitaryRepr 01} as unitary equivalences
\begin{align}
\label{generalization Gamma equation}
\fC \wt \Gamma &= \wt \cV R \bvf \fJ_p \wt \cV^*, \\
\label{generalization Gamma 1 equation}
\fC \wt \Gamma_1 &= \wt \cV R_1 \bvf_1 \fJ_p \wt \cV^*, \\
\label{generalization u equation}
\Gamma^* e_0 &= \wt \cV p .
\end{align}

\begin{rem}
It can be seen from the proof, that the isometry $\cV$ is unique. 

It is also not hard to see that conditions \eqref{e:HankUnitaryRepr 01}, \eqref{e:HankUnitaryRepr 
02} alone guarantee uniqueness of the Hankel operator $\Gamma$, and that in this case, multiplying 
$\cV$ by an appropriate unimodular constant we can get the condition \eqref{e:u UnitaryRepr 01}. 
\end{rem}

The above Theorem \ref{t:AbstrISP} looks like a different (albeit  a more natural) statement  from 
the unitary equivalences \eqref{e:UE01}, \eqref{e:UE02} that were promised in the beginning of this 
subsection. The following simple proposition shows that they are in fact equivalent.

\begin{prop}
\label{p:tuple unitary equivalence}
The identities  \eqref{generalization Gamma equation}, \eqref{generalization Gamma 1 equation}, and 
\eqref{generalization u equation} are equivalent to the unitary equivalence of the tuples $R,R_1, p,
\fJ_p, \bvf, \bvf_1$ and $|\wt 
\Gamma|,|\wt \Gamma_1|,  u, \wt \fJ_u , \wt \Phi, \wt \Phi_1$ \textup(for an appropriate choice of 
$\fJ_u$ in the equivalence class given by \eqref{e:data equiv rel 01}\textup), i.e.~to the 
identities \eqref{e:UE01}, \eqref{e:UE02}. 

In other words, there exists a 
conjugation $\fJ_u$ from the equivalence class given by \eqref{e:data equiv rel 01}, such that the 
identities \eqref{e:UE01}, \eqref{e:UE02} hold. 
\end{prop}

\begin{proof}[Proof of Proposition \ref{p:tuple unitary equivalence}]
Substituting \eqref{e:UE01} and \eqref{e:UE02} into \eqref{e:PolDec04} we immediately get 
\eqref{generalization Gamma equation}, \eqref{generalization Gamma 1 equation}, and 
\eqref{generalization u equation}, so we only need to prove the other direction.
 
 Assuming \eqref{generalization Gamma equation}, \eqref{generalization Gamma 1 equation} and 
 \eqref{generalization u equation}, let us first show that 
\begin{align*}
    |\wt \Gamma| = \wt \cV R \wt \cV^*, \qquad |\wt \Gamma_1| = \wt \cV R_1 \wt \cV^* .
\end{align*}

The modulus $|A|$ of an operator $A$ is the unique self-adjoint non-negative operator such that 
\begin{align*}
\|Ax\|=\|\, |A|x \, \|\qquad \forall x\in\cH. 
\end{align*}
The operator $\wt \cV R \wt \cV^*$ is clearly self-adjoint and non-negative, so, since $\bvf$, 
$\fC$, and $\fJ_p$ are isometries,   we can write
\begin{align*}
\|\wt\Gamma x\| = \|\fC\wt \Gamma x\| = \|\wt\cV \bvf \fJ_p R \wt \cV^* x\| =\| R\wt\cV^* x\| = 
\| \wt\cV R\wt\cV^* x\|;
\end{align*}
in the second equality we also used the fact that  $\bvf$, $\fJ_p$ commute with $R$.

Similarly we can get $|\wt \Gamma_1| = \wt \cV R_1 \wt \cV^*$: 
The fact that $\bvf_1$ is only a 
partial isometry does not spoil anything; since  $\bvf_1$ commutes with $R_1$ and 
$\Ker\bvf_1=\Ker R_1$, we have 
\[
\| \bvf_1 \fJ_p R_1 \wt\cV^*x\| = \| \bvf_1  R_1 \fJ_p \wt\cV^*x\| = \|   R_1 \fJ_p 
\wt\cV^*x\| = \|  \fJ_p R_1  \wt\cV^*x\| = \| R_1  \wt\cV^*x\|, 
\]
and the rest of the computations 
follows exactly as for the case of $|\Gamma|$.

Next we define a conjugation $\wt \fJ_u$ on $(\Ker \Gamma)^{\perp}$ by $\wt \fJ_u := \wt \cV \fJ_p 
\wt \cV^*$. Easy to see that $\wt \fJ_u$ commutes with $|\wt \Gamma| = \wt \cV R \wt \cV^* , |\wt 
\Gamma_1| = \wt \cV R_1 \wt \cV^*$ and preserves $\wt \cV p$. Now we extend $\wt \fJ_u$ to a 
conjugation $\fJ_u$   defining on the whole space $\cH$. This can be done by following the process 
stated in Remark \ref{r: partial conjugation can be extended to a conjugation} (we can set $\wt 
\fJ_u'$ to be an arbitrary partial conjugation with support $\Ker \Gamma$, and then let $\fJ_u = 
\wt \fJ_u + \wt \fJ_u'$, thus we have $\fJ_u$ commutes with $|\Gamma|, |\Gamma_1|$). 

Define $\wt\phi:= \wt \cV \bvf \wt \cV^*$, $\wt\phi_1:= \wt \cV \bvf_1 \wt \cV^*$. Clearly 
$\wt\phi$ is a unitary operator commuting with $|\wt\Gamma|$, and $\wt\phi_1$ is a partial 
isometry, $\Ker\wt\phi_1 =\Ker\wt\Gamma_1$ commuting with $|\Gamma_1|$. 
Now we can rewrite equations \eqref{generalization Gamma equation}, \eqref{generalization Gamma 1 
equation}  as
\begin{align*}
    \wt \Gamma = \fC (\wt \cV R \wt \cV^*)(\wt \cV \bvf \wt \cV^* ) (\wt \cV \fJ_p \wt \cV^*) 
    &= \fC |\wt \Gamma| \wt \phi\, \wt \fJ_u , \\
\wt \Gamma_1 = \fC (\wt \cV R_1 \wt \cV^*)(\wt \cV \bvf_1 \wt \cV^* ) (\wt \cV \fJ_p \wt \cV^*) 
    &= \fC |\wt \Gamma| \wt\phi_1 \wt \fJ_u ,
\end{align*}
which are exactly 
identities \eqref{e:PolDec04} (for the particular choice of $\wt\fJ_u$, $\wt\phi$, $\wt\phi_1$). 

Finally, let us notice that \eqref{generalization u equation} is just the identity $u=\wt\cV p$. 
\end{proof}

Define
\begin{align}
\label{e:q def}
q & := \bvf  R^{-1}p = R^{-1}\bvf p , \qquad \wh q  := \fJ_p q= \bvf^* R^{-1}p = R^{-1}\bvf^*p \\
\label{e:wh Sigma*}
\wh\Sigma^*&:= \fJ_p \Sigma^* \fJ_p = \bvf_1^* R_1 R^{-1} \bvf.
\end{align}
\begin{prop}
\label{p:HankCoeff}
The coefficients $\{\gamma_k\}_{k=0}^\infty$ of the Hankel operator 
$\Gamma=(\gamma_{j+k})_{j,k=0}^\infty$ from Theorem \ref{t:AbstrISP} are given by 
\begin{align}
\label{e:HankCoeff}
\gamma_k &= \la  q,  (\Sigma^*)^k p \ra\ci\cH = \la (\wh \Sigma^*)^k p,  q \ra\ci\cH.  
\end{align}
\end{prop}

\subsection{Some discussions} As it was already discussed above, the \emph{abstract spectral data} 
is the tuple $R,R_1, \fJ_p, \bvf, \bvf_1, p$, described in the beginning of Section 
\ref{s: Abstract ISP setup}, modulo unitary equivalence of the tuples.  It is also natural to 
define the 
abstract spectral data modulo  equivalence relations
\begin{align}
\label{e:EquivRel ASD}
\fJ_p' = \bpsi \fJ_p, \qquad \bvf'= \bvf \bpsi^*, \qquad \bvf_1'=\bvf_1 \bpsi^*, 
\end{align}
between triples $\fJ_p$, $\bvf$, $\bvf_1$ and  $\fJ_p'$, $\bvf'$, $\bvf_1'$; here $\bpsi$ is a 
unitary, $\fJ_p$-symmetric operator,  commuting with $R$ and preserving $p$, $\bpsi p =p$. 

The results presented in the above subsection, imply that the map from a Hankel operator to the  
abstract spectral data (interpreted modulo unitary equivalence and equivalence relations 
\eqref{e:EquivRel ASD})  is injective. 

It is trivially not surjective, because it is quite easy to construct an abstract spectral data 
such that the operator $\Sigma^*$ is not asymptotically stable. What is more non-trivial, it was 
shown in \cite{GPT-22} that even in the case when there no trivial obstacles to the asymptotic 
stability (in \cite{GPT-22} the case of $p$ being cyclic for $R$ was considered), there exists 
abstract spectral data, which does not appear from a Hankel operator. 

The conjugation $\fJ_p$ in the abstract spectral data looks a bit mysterious, especially in light 
of Proposition \ref{p:HankCoeff}, where $\fJ_p$ is absent from the formulas for $\gamma_k$ (while 
$\fJ_p$ appear in the relations between $q$ and $\wh q$, and between $\Sigma^*$ and $\wh\Sigma^*$, 
the formulas for all these object do not involve $\fJ_p$).

\subsection{Proofs}
The whole proof consists of three different parts: existence, uniqueness and the trivial kernel 
 condition.
 
 \subsubsection{Existence of the Hankel operator \tp{$\Gamma$}{}}
\label{s:exists Gamma}
Recalling that $\Sigma^* = \bvf_1 R_1 R^{-1}\bvf^*$ we can write 
\begin{align*}
I-\Sigma\Sigma^* & = \bvf( I - R^{-1}R_1^2 R^{-1}) \bvf^* = \bvf R^{-1}( R^2 -  R_1^2  )R^{-1} 
\bvf^*
\\
&= \bvf R^{-1} p p^* R^{-1} \bvf^* = q q^*, 
\end{align*}
where $q:= \bvf R^{-1}p$ (recall that $\|R^{-1} p\|\le1$). 

Applying both sides of the identity 
\begin{align}
\label{e:defectSigma*}
I-\Sigma\Sigma^* = qq^*, 
\end{align}
to a vector $x\in\cH$ and taking the inner product with $x$, we get 
\begin{align*}
\|x\|^2-\| \Sigma^* x\|^2 = |\la x,q \ra\ci\cH|^2 .
\end{align*}
Replacing $x$ in the above identity by $(\Sigma^*)^k x$, $k=1, 2, \ldots, n-1$, summing the results 
and telescoping, we get that 
\begin{align*}
\|x\|^2 = \|(\Sigma^*)^n x\|^2 = \sum_{k=0}^{n-1} \| (\Sigma^*)^k x\|^2. 
\end{align*}
Letting $n\to \infty$ and using the asymptotic stability of $\Sigma^*$, we conclude that 
\begin{align}
\|x\|^2 = \sum_{k=0}^{\infty} \| (\Sigma^*)^k x\|^2, 
\end{align}
so the operator $\cV:\cH\to \ell^2$, 
\begin{align}
\label{e:cV def}
\cV x := \left( \la x, q\ra, \la \Sigma^* x, q\ra, \la (\Sigma^*)^2 x, q\ra, \ldots  \right) = 
\left( \la (\Sigma^*)^k x, q\ra \right)_{k=0}^\infty, 
\end{align}
is an isometry. The definition \eqref{e:cV def} of $\cV$ immediately implies that $\cV$ intertwines 
$S^*$ and $\Sigma^*$, 
\begin{align}
\label{e:Sigma* S* intertwine}
S^*\cV = \cV \Sigma^* .
\end{align}


Now we define the  operators $\Gamma$ and $\Gamma_1$ as 
\begin{align}
\label{e:construct Gamma Gamma1}
\Gamma := \fC \cV R \bvf \fJ_p \cV^*  \qquad \Gamma_1 := \fC \cV R_1 \bvf_1 \fJ_p \cV^*.
\end{align}

We will show that $\Gamma$ is a Hankel operator by proving that $\Gamma S = \Gamma_1 =S^* \Gamma$.

To show that $\Gamma S = \Gamma_1$,  recall that $\Sigma = \bvf  R^{-1} R_1 \bvf_1^*$  and 
define $\wh\Sigma:= \bvf^* R^{-1} R_1 \bvf_1$ (both operators are 
well defined, see Lemma \ref{l:Douglas}).  It is easy to see that  
that $\wh\Sigma \fJ_p = \fJ_p \Sigma$; together with identity $\cT^* \cV^*=\cV^* S$ (which is just 
the adjoint of \eqref{e:Sigma* S* intertwine})  it gives us
\begin{align*}
    \Gamma_1 &= \fC \cV R_1 \bvf_1 \fJ_p \cV^* = \fC \cV R \bvf \wh\Sigma \fJ_p \cV^* \\
    &= \fC \cV R \bvf \fJ_p \Sigma \cV^* = \fC \cV R \bvf \fJ_p \cV^* S \\
    &= \Gamma S .
\end{align*}
As for the identity $\Gamma_1 =S^*\Gamma$, recalling that $\Sigma^* = R_1 \bvf_1 R^{-1} 
\bvf^*$ 
and using \eqref{e:Sigma* S* intertwine}, we have
\begin{align*}
    \fC \Gamma_1 &= \cV R_1 \bvf_1 \fJ_p \cV^* = \cV \Sigma^* R \bvf \fJ_p \cV^* \\
    &= S^* \cV R \bvf \fJ_p \cV^* = S^* \fC \Gamma .
\end{align*}
Since $\fC$ commutes with $S^*$, we see that $\fC\Gamma_1 = \fC S^*\Gamma$, and left multiplying 
this identity by $\fC$ we get the desired result. 

Thus we know $\Gamma$ is a Hankel operator. Identities  
\eqref{e:HankUnitaryRepr 01}, \eqref{e:HankUnitaryRepr 02} are trivial; we get them by applying 
$\fC$ to both sides of identities \eqref{e:construct Gamma Gamma1}.

Let us prove  \eqref{e:u UnitaryRepr 01}. Since $\Gamma$ is $\fC$-symmetric, we get from 
\eqref{e:construct Gamma Gamma1}
\begin{align*}
\Gamma^*  = \fC \Gamma \fC = \cV R \bvf \fJ_p \cV^* \fC . 
\end{align*}
It is easy to see that $\cV^*e_0=q$. Since $\fC e_0 =e_0$, we get 
\begin{align*}
\Gamma^* e_0 = \cV R \bvf \fJ_p \cV^* e_0 = \cV R \bvf \fJ_p q = \cV R \fJ_p \bvf^*q   = \cV  \fJ_p 
R 
\bvf^* q = \cV \fJ_p p =\cV p; 
\end{align*}
here in the third equality we used the fact that $\bvf$ is $\fJ_p$-symmetric, in the next one that 
$\fJ_p$ commutes with $R$, and in the last one that $\fJ_p p =p$. \hfill\qed

Note, that if only conditions \eqref{e:HankUnitaryRepr 01}, \eqref{e:HankUnitaryRepr 02} are 
satisfied, we get from rank one perturbation relations $|\Gamma|^2-|\Gamma_1|^2 = uu^*$, 
$R^2-R_1^2=pp^*$ that 
\begin{align*}
uu^* = \cV pp^*\cV^*, 
\end{align*}
which happens if and only if $\cV p =\alpha p$, $|\alpha|=1$. Requiring that $\cV p=p$ we pick the 
unique operator $\cV$. As it was shown above in the proof, this operator is given by \eqref{e:cV 
def}.

\subsubsection{Uniqueness of the Hankel operator \tp{$\Gamma$}{} and formula for the coefficients.}
We know that Hankel coefficients $\gamma_k$ are given by
\begin{align*}
\gamma_k = \la\Gamma e_k , e_0\ra = \la e_k , \Gamma^* e_0\ra = \la S^k e_0 , u\ra = \la  e_0 , 
(S^*)^k u\ra .
\end{align*}
Since $u\in (\ker\Gamma)^\perp$ and $(\ker\Gamma)^\perp$ is an 
$S^*$-invariant subspace of $\ell^2$, we can replace $S^*$ by $S^*\big|_{(\ker\Gamma)^\perp}$ and 
$e_0$ by  $v:=P\ci{(\ker\Gamma)^\perp} e_0$, so 
\begin{align*}
\gamma_k = \bigl\la v , (S^*\big|_{(\ker\Gamma)^\perp})^k u \bigr\ra .
\end{align*}
We know that the conditions \eqref{e:HankUnitaryRepr 01}, \eqref{e:HankUnitaryRepr 02}, and 
\eqref{e:u UnitaryRepr 01} imply the unitary equivalences \eqref{e:UE01}, \eqref{e:UE01}, see 
Proposition \ref{p:tuple unitary equivalence}. These unitary equivalences and the fact that 
$S^*\big|_{(\ker\Gamma)^\perp} = \Phi_1 |\Gamma_1| |\Gamma|^{-1} \Phi^*$  imply that 
$\Sigma^*
$ from Theorem \ref{t:AbstrISP} satisfies
\begin{align*}
S^*\big|_{(\ker\Gamma)^\perp} = \wt\cV \Sigma^* \wt\cV^*, 
\end{align*}
so 
\begin{align*}
\gamma_k = \bigl\la v , \wt\cV (\Sigma^*)^k \wt\cV^* u \bigr\ra 
= \bigl\la \wt\cV^* v ,  (\Sigma^*)^k \wt\cV^* u \bigr\ra 
= \bigl\la \wt\cV^* v ,  (\Sigma^*)^k p \bigr\ra .
\end{align*}
Now it looks like that to prove the first part of \eqref{e:HankCoeff} it is sufficient to recall 
that $\wt \cV^* v =\cV^* e_0 =q$, which was already proved in the above Subsection \ref{s:exists 
Gamma}. But unfortunately, this was proved only for the operator $\cV$ given by \eqref{e:cV def}, 
but here we cannot assume that $\cV$ is given by this formula. Thus a more complicated reasoning is 
necessary. 

Since $\Gamma = \fC|\Gamma| \Phi \fJ_u $, and $\Gamma$ is $\fC$-symmetric
\begin{align*}
\Gamma^* =\fC\Gamma\fC = |\Gamma| \Phi \fJ_u \fC = |\Gamma|  \fJ_u \Phi^*  \fC 
= \fJ_u |\Gamma|   \Phi^*  \fC, 
\end{align*}
and so 
\begin{align*}
u=\Gamma^*e_0 = \fJ_u|\Gamma|   \Phi^*  \fC e_0 = \fJ_u |\Gamma|   \Phi^*   e_0. 
\end{align*}
Since $\fJ_u u =u$, we conclude that $u= |\Gamma|   \Phi^*   e_0$.

The operators $\Phi^*$,  $|\Gamma|$ are $0$ on $\Ker\Gamma$, so restricting everything 
to $(\ker\Gamma)^\perp$ we get 
\begin{align*}
u =  |\wt\Gamma|   \wt\Phi^* v = \wt\Phi^* |\wt\Gamma|    v
\end{align*}
so $v=  \wt\Phi|\wt\Gamma|^{-1} u$. The unitary equivalences \eqref{e:UE01}, \eqref{e:UE02} imply 
then that $\wt\cV^* v =\bvf R^{-1}p =:q $, and the first part of \eqref{e:HankCoeff} is proved. 

As for the second part, 
\begin{align*}
\la q, (\Sigma^*)^k p \ra = \la \fJ_p (\Sigma^*)^k p , \fJ_p q\ra 
=\la \fJ_p (\Sigma^*)^k \fJ_p p , \fJ_p q\ra = \la  (\wh\Sigma^*)^k  p , \wh q\ra .
\end{align*}

Since the formula for $\gamma_k$ involves only the abstract spectral data $R$, $R_1$, $\bvf$, 
$\bvf_1$, the Hankel operator $\Gamma$ is unique. 
\hfill\qed

\subsubsection{The trivial kernel condition}
\label{s:TrivKer}
It is easy to see that  $\Ker \Gamma = \{ 0 \}$ if and only if $\cV: \cH \to \ell^2$ defined by  
\eqref{e:cV def} is a unitary operator, i.e.~if and only if  $\Ran \cV = \ell^2$. 
      
So, if $\Ker \Gamma = \{ 0 \}$, then $\Sigma^*$ is unitarily equivalent to the backward shift 
$S^*$, and comparing identities
\begin{align*}
     I-SS^* = e_0 e_0^* , \qquad I-\Sigma\Sigma^* = qq^*,
\end{align*}
where, recall $q= \bvf R^{-1}p$, we conclude that $\|R^{-1}p\|=\| q \|  = 1 $. We can also see that 
the identity $I-\Sigma\Sigma^* = qq^*$ implies that $\Sigma^*q=0$. 

Also in this case, since $\Gamma_1 =\Gamma S$, we can conclude that $\ker\Gamma_1=\{0\}$; thus 
$\ker R_1=\{0\}$, and therefore $\bvf_1$ is a unitary operator.

Now, assuming $q \in \Ran R$, let us get   a contradiction. 
      
Let $q=Rx$, $x\in\cH$. Define $f :=\bvf_1 \bvf^*   x$, thus $x =  \bvf 
\bvf_1^* f$ and $ q = R \bvf \bvf_1^* f$, Hence
\begin{align*}
    \Sigma^* R \bvf \bvf_1^* f = \Sigma^*  q = 0.
\end{align*}
But on the other hand we have
\begin{align*}
     \Sigma^* R \bvf \bvf_1^* =\bvf_1 R_1 R^{-1} R \bvf^* \bvf  \bvf_1^*   = R_1 = \bvf_1 
     \bvf^* 
     R \Sigma ,
\end{align*}
where the last equality follows from $R_1$ being a self-adjoint operator. Hence  
\begin{align*}
    \bvf_1 \bvf^* R \Sigma f = 0 , 
\end{align*}
so $\Sigma f = 0$ (because the operators $R$, $\bvf^*$, $\bvf$ have trivial kernels), which 
contradicts to the fact that $\Sigma$ is an isometry. Hence $q \notin \Ran 
R$.

Now we prove the sufficiency part. Suppose $\| q \| = 1$ and $q \notin \Ran R$. 

Let us first show that 
$\Ker R_1 = \{ 0 \}$.
Let $ R_1 x = 0 $ for a $x\ne0$.  Applying to $x$ the identity
\begin{align*}
   R_1^2 = R \big( I- \bvf^* q (\bvf^* q)^* \big) R,
\end{align*}
we get that 
\begin{align*}
   R x = \langle R x,\bvf^* q \rangle \bvf^* q;
\end{align*}
note that $Rx\ne 0$ because $R$ has trivial kernel. 
This implies $R  x  =  \alpha \bvf^* q$, $|\alpha| = 1$, so
\begin{align*}
q=\alpha^{-1} \bvf Rx = \alpha^{-1}R\bvf x, 
\end{align*}
which contradicts  the assumption $q \notin \Ran R$. So, indeed, $\ker R_1 =\{0\}$. Note that in 
this case $\bvf_1$ is a unitary operator. 

From the definition of $\Sigma^*$ we get that $\Sigma = \bvf R^{-1} R_1 \bvf_1^*$, is well defined 
for all $x\in\cH$, see Lemma \ref{l:Douglas}. Since all operators have trivial kernels, we conclude 
that $\ker\Sigma=\{0\}$.

Applying both sides of the identity \eqref{e:defectSigma*} to $q$, we get that $\Sigma\Sigma^* 
q=0$, so $\Sigma^*q=0$.

Now left and right multiplying \eqref{e:defectSigma*} by $\Sigma $ and $\Sigma^*$ respectively, we 
get
\begin{align*}
\Sigma^* \Sigma - \Sigma^* \Sigma \Sigma^* \Sigma   = \Sigma^* \big( qq^*\big) \Sigma 
= \Sigma^*q (\Sigma^* q)^*= 0,
\end{align*}
hence $\Sigma^* \Sigma$ is a projection. Furthermore, since $\Ker \Sigma = \{ 0 \}$, we have 
$\Sigma^* \Sigma = I$, and so $\Sigma$ is an isometry.

Since $\Sigma$ is an isometry, $\|q\|=1$,  and $\Sigma^* q=0$, it is an easy exercise to show that 
$\{ \Sigma^k q\}_{k=0}^\infty$ is an orthonormal system. Rewriting decomposition \eqref{e:cV def} 
for $\cV$ as 
\begin{align*}
\cV x = \left( \la  x, \Sigma^k q\ra \right)_{k=0}^\infty \in\ell^2, 
\end{align*}
we immediately see that $\cV$ is surjective. It was already proved that $\cV$ is an isometry, so 
$\cV$ is unitary. 
\hfill\qed

\subsection{Some remarks}

\section{Concrete representations of the abstract spectral data for simple cases}
\label{s:Concr Repr of ASD}
The abstract spectral data treated in Section \ref{s: Abstract ISP}  looks like some completely 
non-tractable abstract nonsense. However in many interesting situations the abstract spectral data 
can be represented using very concrete and understandable objects. 

\subsection{The case of \tp{$p$}{p} being cyclic for \tp{$R$}{R}} If $p$ is a cyclic vector for 
$R$, the pair $R$, $p$ is completely determined by the spectral measure $\rho=\rho\ci{R,p}$ of the 
operator $R$ with respect to the vector $p$. 

In this case the pair $R$, $p$ can be represented as (unitarily equivalent to)  the multiplication 
operator $M$, $M f(s) =sf(s)$ in $\cH=L^2(\rho)$ and the vector $p= \1\in L^2(\rho)$.  Under these 
assumptions there is a unique conjugation $\fJ_p$, commuting with $R$ and preserving $p$, namely 
$\fJ_p f = \bar f$. 

The unitary operator $\bvf$ is given by the multiplication by an appropriate unimodular function, 
which we  will be denoting  by $\vf$ (without boldface). Thus we could write $\bvf=\vf(R)$. 

Note, that the condition $\|R^{-1}p\| \le 1$ is translated to 
\begin{align}
\label{e:normalization of rho}
\int_\R \frac{\dd\rho(s)}{s^2} \le 1
\end{align}

It is easy to see that if $p$ is cyclic for $R$ (and so for $R^2$), it is also cyclic for $R_1^2 
=R^2-pp^*$, and so for $R_1$. The spectral measure $\rho^{[1]} = \rho\ci{R_1,p}$ is uniquely 
determined from $\rho$, and the partial isometry  $\bvf_1$ is defined by a function $\vf_1$, 
$\vf_1(s)|= 1$ $\rho^{[1]}$-a.e.~on $(0,\infty)$, and $\vf_1(0)=0$, and we can write $\bvf_1 = 
\vf_1(R_1)$.

Thus, the abstract spectral data in this case is given by
\begin{enumerate}
\item A  compactly supported measure $\rho$ on $(0, \infty)$ satisfying \eqref{e:normalization of 
rho} (and the  measure $\rho^{[1]}$ derived from it); 

\item Two unimodular functions $\vf$ and $\vf_1$ on $(0,\infty)$, defined $\rho$-a.e.\ and 
$\rho^{[1]}$-a.e.\ respectively. 
\end{enumerate}
The conjugation $\fJ_p$ is implicit in this case. 

This case was studied in details in \cite{GPT-22}, where the asymtotic stability of the operator 
$\Sigma^*$ was investigated. There is no obvious obstacles for the asymptotic stability of 
$\Sigma^*$ in this situation, and it was conjectured for some time, that in this case the operator 
$\Sigma^*$ is always asymptotically stable. However, it was shown in \cite{GPT-22} that this is not 
true; some counterexamples as well as some sufficient conditions for the asymptotic stability were 
presented there, but describing these result is out of the scope of the current paper.  

We will only mention that the conjugate linear Hankel operator $H_u$, studied in \cite{GPT-22} 
coincides with our operator $\fC\Gamma = \Gamma^*\fC$, treated as the operator in the Hardy space 
$H^2$. The symbol $u$ in \cite{GPT-22} was in our terms just the function $u(z) = \sum_{k\ge0} 
\bar\gamma_k z^k$, which is just our vector $u=\Gamma^* e_0$ transferred to $H^2$ via 
$z$-transform. Our functions $\vf$, $\vf_1$ correspond to the functions $\Psi^*$ and $\wt\Psi$ from 
\cite{GPT-22}.

\subsection{The case of a compact operator \tp{$R$}{R} with cyclic \tp{$p$}{p}}
\label{case of compact operator with cyclic vector}

In this case the setup can simplified even more. The spectral measures $\rho$ and $\rho^{[1]}$ in 
this case can be written down as 
\begin{align}
\label{e:repr rho compact case}
\rho = \sum_{k\ge1} w_k \delta_{\lambda_k}, \qquad \rho^{[1]} = \sum_{k\ge1} w_k^{[1]} 
\delta_{\mu_k}
\end{align}
where $\lambda_1>\lambda_2>\ldots> \lambda_k> \ldots> 0$ are eigenvalues of $R$, 
$\mu_1>\mu_2>\ldots> \mu_k> \ldots\ge 0$  are eigenvalues of $R_1$ and $w_k = 
\rho(\{\lambda_k\})$, $w_k^{[1]} = \rho^{[1]}(\{\mu_k\})$. Recall, see Proposition \ref{r:ker R_1}, 
that the operator $R_1$ can have nontrivial kernel. In the infinite-dimensional case we denote this 
zero eigenvalue by $\mu_\infty$, so $\mu_k>0$ for $k\in\N$; if $\dim\cH=n$, then $\mu_n\ge0$ 
and $\lambda_n>0$.

  Note that the eigenvalues $\lambda_k$, $\mu_k$ satisfy the intertwining 
relations 
\begin{align}
\label{e:Intertwine lambda mu}
\lambda_1>\mu_1>\lambda_2>\mu_2 \ldots> \lambda_k> \mu_k > \ldots\ge 0, 
\end{align}
and $\lambda_k$ is either a finite sequence $\lambda_1, \lambda_1, \ldots, \lambda_n$ (in this case 
$\mu_n$ can be $0$), or $\lambda_k \searrow 0$. 

The intertwining relation \eqref{e:Intertwine lambda mu} can be easily obtained from the minimax 
characterization of the eigenvalues, or via standard formulas for rank one perturbations relating 
Weil's $m$-functions, see Remark \ref{r:intertwine}. We leave details for the reader. 

The only new twist in the compact case is that the weights $w_k$ (and therefore $w^{[1]}_k$) can be 
determined from the eigenvalues $\lambda_k$ of $R$ and non-zero eigenvalues $\mu_k$ of $R_1$. The 
key to that is the following \emph{abstract Borg's theorem}. 

\begin{thm}[Abstract Borg's Theorem]
\label{t:Borg 01}
Given two sequences $\{\lambda_k\}_{k\ge 1}$ and $\{\mu_k\}_{k\ge 1}$ satisfying intertwining 
relations \eqref{e:Intertwine lambda mu} and such that $\lambda_k^2\to 0$ as $k\to \infty$, there 
exists a unique \tup{(}up to unitary equivalence\tup{)} triple $(W,W_1,p)$, such that 
\begin{enumerate}
\item  $W=W^*\ge 0$, $\Ker W=\{0\}$ is a 
compact operator with simple eigenvalues $\{ \lambda_k^2\}_{k=1}^{\infty}$;
\item $p\in\cH$ and $W_1 = W-pp^*$ is a compact operator with non-zero eigenvalues $\{ \mu_{k}^2 
\}_{k=1}^{\infty}$  
\tup{(}$W_1$ can 
also have a simple eigenvalue at $0$, and it is not hard to show that all the eigenvalues are 
simple\tup{)}.
\end{enumerate}
Moreover, $\|W^{-1/2} p\|=1$ if and only if 
\begin{align}
\label{e: norm q = 1 02}
\prod_{k\ge1}\frac{\mu_k^2}{\lambda_k^2}=0,\quad \text{ or, equivalently, } \quad \sum_{k\ge1} 
\Bigg( \frac{\lambda_{k}^{2}}{\mu_{k}^{2}} -1  \Bigg) = \infty; 
\end{align}
in addition, if \eqref{e: norm q = 1 02} holds, then  $\|W^{-1} p\|=\infty$ if and only if 
\begin{align}
\label{e: q not in ran R 02}
\sum_{k\ge1}^\infty \left( \frac{\mu_k^2}{\lambda_{k+1}^2} -1 \right) =\infty .
\end{align}
\end{thm}

Applying this theorem, from given sequences $\{\lambda_k\}_{k\ge 1}$, $\{\mu_k\}_{k\ge 1}$ we 
get the (unique up to unitary equivalence) triple $W$, $W_1$, $p$, and defining $R=W^{1/2}$, 
$R_1=W_1^{1/2}$ (non-negative square roots) we get the unique triple $R$, $R_1$, $p$. 

Thus, the spectral data in this case is given by sequences $\{\lambda_k\}_{k\ge 1}$, 
$\{\mu_k\}_{k\ge 1}$ of eigenvalues of $R$ and $R_1$ and the sequences $\{\xi_k\}_{k\ge 1}$, 
$\{\eta_k\}_{k\ge 1}$, $|\xi_k|=|\eta_k|=1$ representing the functions $\vf$ and $\vf_1$, 
$\xi_k=\vf(\lambda_k)$, $\eta_k=\vf_1(\mu_k)$. 

It will be shown below, see Section \ref{asymptotic stability}, that the operator $\Sigma^*$ in the compact case is 
always asymptotically stable, so for any abstract spectral data there exist a unique Hankel 
operator. 

Condition $\|W^{-1/2}p\|=1$ is equivalent to the condition $\|R^{-1}p\|=1$, and the condition  
$\|W^{-1}p\|=\infty$ means $R^{-1}p\notin \ran R$. Therefore, for the resulting Hankel operator 
$\Gamma$, $\ker \Gamma=\{0\}$ if and only if conditions \eqref{e: norm q = 1 02} and \eqref{e: q 
not in ran R 02} hold. 

\begin{rem}
\label{r:AbstrBorg finite} As it will be seen from the proof, the abstract Borg's theorem also 
holds  (and is trivial)
for finite rank operators, i.e.~for finite sequences $\{\lambda_k\}_{k=1}^n$, $\{\mu_k\}_{k=1}^n$. 
In this case 
\begin{align*}
\lambda_1>\mu_1>\lambda_2>\mu_2>\ldots \lambda_1>\mu_n\ge 0, 
\end{align*}
and $\{\lambda_k\}_{k=1}^n$, $\{\mu_k\}_{k=1}^n$ are exactly the eigenvalues of $W$ and $W_1$. 

Condition \eqref{e: norm q = 1 02} holds in this case if and only if $\mu_n=0$, and the condition 
\eqref{e: q not in ran R 02} is never satisfied. This agrees with the simple fact that a finite 
rank Hankel operator always have a non-trivial kernel.  
\end{rem}

\section{The abstract spectral data for general compact operators} 
\label{s:ASD compact}
Consider now the case when $R$ 
is a general 
compact operator. Let $\cH_0:= \cspn\{R^n p: n\ge0 \}$, and let $\{\lambda_k\}_{k\ge1}$ be the 
eigenvalues of $R\big|_{\cH_0}$ taken in decreasing order, and let $\{\mu_k\}_{k\ge1}$ be the 
\emph{non-zero} eigenvalues of $R_1\big|_{\cH_0}$, also taken in the decreasing order (note that 
$R_1\big|_{\cH_0}$ can also have a simple eigenvalue at $0$). Trivially the eigenvalues $\lambda_k$ 
and $\mu_k$ satisfy the intertwining relations \eqref{e:Intertwine lambda mu}. 

The spectral measures $\rho$ of $R\big|_{\cH_0}$ and $\rho^{[1]}$ of $R_1\big|_{\cH_0}$ can be 
represented as in the case of simple spectra by formulas \eqref{e:repr rho compact case}. It 
follows from the Abstract Borg's Theorem (Theorem \ref{t:Borg 01}) that the measures $\rho$ and 
$\rho^{[1]}$ can be reconstructed from the non-zero eigenvalues $ \{ \lambda_k \}_{k \ge 1}$, $\{ \mu_k \}_{k \ge 1}$. 

So, the sequences $\{\lambda_k\}_{k\ge1}$, $\{\mu_k\}_{k\ge1}$  give us a part of the spectral 
data. 

Let us analyze how to get the full spectral data. Define
\begin{align}
\label{e:K K_1}
\cK & :=\cspn\{ \ker (R -\lambda_k I) : k\ge 1 \} \ominus \cH_0, \\ 
\label{e:K_1 alt 01}
\cK_1 & :=\cspn\{ \ker (R_1 -\mu_k I) : k\ge 1 \} \ominus \cH_0
\end{align}

Trivially $\cH_0$ is an invariant subspace for both $R^2$ and $R_1^2 = R^2-pp^*$, and so for both
$R$ and $R_1$. Since
\begin{align*}
\cK =\cspn\{ \ker (R -\lambda_k I) : k\ge 1 \} \cap \cH_0^\perp, 
\end{align*}
we see that $\cK$ is invariant for $R$, as the intersection of invariant subspaces. Finally, since 
$R$ and $R_1$ coincide on $\cH_0^\perp$, we conclude that $\cK$ is also invariant for $R_1$. 

Similarly, we can conclude that $\cK_1$ is invariant for both $R_1$ and $R$. 

Trivially $\cK\perp\cK_1$. 

\begin{lm}
\label{l:spectrum R}
Let $R$, $R_1$, $p$, $\fJ_p$, $\bvf$, $\bvf_1$ be an abstract spectral data such  that the operator 
$\Sigma^* 
= \bvf_1 R_1 R^{-1} 
\bvf^* =  R_1 \bvf_1 \bvf^* R^{-1} $ is asymptotically stable. 
Then 
\begin{align*}
\cH= \cH_0\oplus \cK \oplus\cK_1, 
\end{align*}
or, equivalently, $\lambda_k$, $\mu_k$ are the only possible eigenvalues of $R$. 
\end{lm}

\begin{rem*}
Note, that generally not all $\mu_k$ are eigenvalues of $R$. One can see from \eqref{e:defect R} 
below, that $\mu_k$ is an eigenvalue of $R$ if and only if $\dim\ker (R_1-\mu_k I) \ge 2$. 
\end{rem*}

\begin{proof}[Proof of Lemma \ref{l:spectrum R}]
If $\{0\} \ne \cH_1 := (\cH_0 \oplus\cK\oplus \cK_1)^\perp$, then the operator $R\big|_{\cH_1}$ 
(and so $R$) has an eigenvalue $s$, different from all $\lambda_k$, $\mu_k$. Since $R$ and $R_1$ 
coincide on $\cH_0^\perp$, we have $\ker (R- sI ) = \ker (R_1 - sI )$. 

Since $\bvf$ and $\bvf_1$ commute with $R$ and $R_1$ respectively, the subspace $\ker (R- sI ) $ is 
a reducing subspace for both  $\bvf$ and $\bvf_1$. Therefore $\ker (R- sI ) $ is a reducing 
subspace for the operator $\Sigma^*= \bvf_1 R_1 R^{-1} \bvf^*$ and $\Sigma^*$ acts unitarily there, 
which contradicts asymptotic stability.  
\end{proof}


%

\subsection{Structure of eigenspaces}
Denote
\begin{align*}
p_k:= P\ci{\ker (R-\lambda_k I)} p, \qquad p^1_k:= P\ci{\ker (R_1-\mu_k I)} p .
\end{align*}
It follows from the spectral theorem that the mapping $f\mapsto f(R) p$ is a unitary map from 
$L^2(\rho)$ to $\cH_0$, so 
\begin{align}
\label{e:norm p_k}
p_k\in\cH_0 \quad\text{and}\quad &\|p_k\|^2 = \rho(\{\lambda_k\}) = w_k >0;
\intertext{similarly}
\label{e:norm p^1_k}
p_k^1\in\cH_0 \quad\text{and}\quad &\|p_k^1\|^2 = \rho^{[1]}(\{\mu_k\}) = w^{[1]}_k >0 .
\end{align}

\begin{lm}
\label{l:struct of evalues}
We have
\begin{align}
\label{e:defect R_1}
\ker (R-\lambda_k I) &= \ker(R_1-\lambda_k I) \oplus \spn\{p_k\}, \\
\label{e:defect R}
\ker (R_1-\mu_k I) &= \ker(R-\mu_k I) \oplus \spn\{p^1_k\} .
\end{align}
\end{lm}
\begin{proof}
The proof follows immediately from the fact that in the decomposition $\cH=\cH_0 \oplus \cK 
\oplus\cK_1$ both $R$ and $R_1$ have the block diagonal structure. 
\end{proof}

We can also get different formula for $\cK$, $\cK_1$
\begin{align}
\label{e:K alt}
\cK & = \cspn\{ \ker (R_1 -\lambda_k I) : k\ge 1 \} \\
\label{e:K_1 alt}
\cK_1 & = \cspn\{ \ker (R  -\mu_k I) : k\ge 1 \}.
\end{align}
\subsection{Canonical choice of the conjugation \tp{$\fJ_p$}{Jp}}
\label{s:canonical J_p}

Let $\fJ_p$ be a conjugation commuting with $R$ and preserving $p$. By Lemma \ref{l:Jp} $\fJ_p$ is 
a reducing subspace for $\fJ_p$, meaning that both $\cH_0$ and $\cH_0^\perp$ are invariant for 
$\fJ_p$.

The conjugation $\fJ_p$ commutes with both $R$ and $R_1$, so eigenspaces of both operators should 
be invariant for $\fJ_p$. Therefore by \eqref{e:K alt}, \eqref{e:K_1 alt} subspaces $\cK$, $\cK_1$ 
are invariant for $\fJ_p$. 

Since the unitary operator $\bvf$ commutes with $R$, eigenspaces of $R$ should be reducing 
subspaces for $\bvf$, so by \eqref{e:K_1 alt} $\cK_1$ is a reducing subspace for $\bvf_1$. 
Similarly, \eqref{e:K  alt} implies that $\cK$ is a reducing subspace for  $\bvf_1$. 

Thus, by picking the unitary operator $\bpsi := I\ci{\cH_0} \oplus \bvf_1\big|_{\cK} \oplus 
\bvf\big|_{\cK_1}$, we find a representative from the equivalence class give by \eqref{e:EquivRel 
ASD} such that 
\begin{align}
\label{e:canonical phi}
\bvf\big|_{\cK_1} = I\ci{\cK_1}, \qquad \bvf_1\big|_{\cK} = I\ci{\cK} \,.
\end{align}

\begin{lm}
\label{l:cyclicity 01}
Let $R$, $R_1$, $p$, $\fJ_p$, $\bvf$, $\bvf_1$ be an abstract spectral data such that $\bvf$ and  
$\bvf_1$ satisfy conditions \eqref{e:canonical phi}. If the operator $\Sigma^* = \bvf_1 R_1 R^{-1} 
\bvf^* =  R_1 \bvf_1 \bvf^* R^{-1} $ is asymptotically stable, then for each $k$ the vector $p_k$ 
is $*$-cyclic for $\bvf\big|_{\ker (R-\lambda_k I)}$, and  $p_k^1$ is $*$-cyclic for 
$\bvf_1\big|_{\ker (R_1-\mu_k I)}$. 
\end{lm}
\begin{proof}
Let for some $k$ the vector $p_k$ is not $*$-cyclic for $\bvf\big|_{\ker (R-\lambda_k I)}$. Then 
there exists a subspace $\cE\subset \ker (R-\lambda_k I)$, $p_k\perp \cE$, which is reducing for 
$\bvf\big|_{\ker (R-\lambda_k I)}$ (and thus for $\bvf$). Since $\cE\perp p_k$, we can also say 
that $\cE\subset \ker (R_1-\lambda_k I)$. 

By \eqref{e:canonical phi} the  operator $\bvf_1$ acts as identity on $\ker (R-\lambda_k I)$. Also, 
trivially
\begin{align*}
R\big|_{\ker (R-\lambda_k I)} = \lambda_k I\ci{\ker (R-\lambda_k I)}, \qquad
R_1\big|_{\ker (R_1-\lambda_k I)} = \lambda_k I\ci{\ker (R_1-\lambda_k I)}. 
\end{align*}
Since $\cE \subset \ker (R_1 - \lambda_k I) \subset 
\ker (R - \lambda_k I)$ and $\cE $ is reducing for $\bvf^*$, we conclude that $\cE$ is an invariant 
subspace for $\Sigma^*$ and $\Sigma^*\big|_\cE=\bvf^*\big|_\cE$, so it is unitary. 
But in this case $\Sigma^*$ cannot be asymptotically stable, and we got a contradiction. Thus, 
$p_k$ is $*$-cyclic for $\bvf\big|_{\ker (R-\lambda_k I)}$. 

The fact that $p_k^1$ is $*$-cyclic for $\bvf_1\big|_{\ker (R_1-\mu_k I)}$ is proved exactly the 
same way, with obvious changes. 
\end{proof}

So, for the abstract spectral data we have a canonical choice of the unitary operators $\bvf$ and 
$\bvf_1$, commuting with $R$ and $R_1$ respectiely: they must satisfy conditions \eqref{e:canonical 
phi}, and the vectors $p_k$ and $p_k^1$ should be $*$-cyclic for $\bvf\big|_{\ker (R-\lambda_k I)}$ 
and $\bvf_1\big|_{\ker (R_1-\mu_k I)}$ 
respectively.%
\footnote{the $*$-cyclicity condition must be satisfied, since otherwise $\Sigma^*$ is not 
asymptotically stable, and the abstract spectral data does not correspond to a Hankel operator.}

So, suppose we have such unitary operators $\bvf$ and $\bvf_1$. Do they indeed correspond to an
abstract spectral data, i.e.~can find a conjugation $\fJ_p$, commuting with $R$ and preserving $p$,
such that both $\bvf$ and $\bvf_1$ are $\fJ_p$-symmetric? And if such conjugation $\fJ_p$ exists, is
it unique?

\begin{lm}
\label{l:uniq J_p}
Let $\bvf$ and $\bvf_1$ be unitary operators, commuting with $R$ and $R_1$ respectively and 
satisfyng the conditions \eqref{e:canonical phi}. Assume also that $p_k$ is $*$-cyclic%
\footnote{For unitary operators in a finite-dimensional space, or more generally, for unitary 
operators with purely singular spectral measure, a vector $p$ is cyclic if and only if it is 
$*$-cyclic.   However, we do not want to bother readers with the unnecessary details, so we state 
this lemma using the notion of $*$-cyclicity.}
for 
$\bvf\big|_{\ker (R-\lambda_k I)}$ and $p_k^1$ is $*$-cyclic for $\bvf_1\big|_{\ker (R_1-\mu_k 
I)}$. 

Then there exists a unique conjugation $\fJ_p$, commuting with $R$ and preserving $p$, such that 
both $\bvf$ and $\bvf_1$ are $\fJ_p$-symmetric. 
\end{lm}
To prove Lemma \ref{l:uniq J_p}, we need the following simple lemma, that is true for arbitrary 
unitary operator. 

\begin{lm}
\label{l:uniq Conj 01} Let $\bu$ be a unitary operator, and let $p$ be a $*$-cyclic vector for 
$\bu$. 
There exists a unique conjugation $\fJ$, preserving $p$ and such that $\bu$ is $\fJ$-symmetric. 

This conjugation is defined by 
\begin{align}
\label{e:def conj 02}
\fJ \left( \sum_{k\in\Z} \alpha_k \bu^k p \right) = \sum_{k\in\Z} \bar\alpha_k \bu^{-k} p 
\end{align}
on the dense set of finite linear combinations $\sum_{k\in\Z} \alpha_k \bu^k p$. 
\end{lm}
\begin{rem}
\label{r:J in spectral repr}
Let $\rho$ be the spectral measure of the operator $\bu$ from the above lemma, corresponding to 
the 
vector $p$. It is easy to see that in the spectral representation of $\bu$, where $\bu$ 
corresponds to 
the multiplication operator by the independent variable $\xi$, and $p$ corresponds to $\1\in 
L^2(\rho)$, the operator $\fJ$ from Lemma \ref{l:uniq Conj 01} is given by 
\begin{align*}
\fJ f = \overline f , \qquad f\in L^2(\rho). 
\end{align*}
\end{rem}

\begin{proof}[Proof of Lemma \ref{l:uniq Conj 01}]
Let $f$ be a finite linear combination $f=\sum_{k\in\Z} \alpha_k \bu^k p$. Then 
\begin{align*}
\fJ \bu f & = \fJ \left( \sum_{k\in\Z} \alpha_k \bu ^{k+1} p \right) =\sum_{k\in\Z} 
\bar\alpha_{-k} 
\bu ^{-k-1} p  \\
& =\bu ^* \sum_{k\in\Z} \bar\alpha_{-k} \bu ^{-k} p = \bu ^* \fJ f, 
\end{align*}
so the identity $\fJ \bu  = \bu ^* \fJ$ holds on a dense set. By continuity it holds on the whole 
space 
$\cH$, so $\bu $ is $\fJ$-symmetric. 

On the other hand, if $\bu $ is $\fJ$-symmetric, and $\fJ p=p$, then $\fJ \bu ^k p = \bu ^{-k} 
\fJ p = 
\bu ^{-k}  p$, so by conjugate linearity of $\fJ$ we get \eqref{e:def conj 02}. So the conjugation 
$\fJ$ is uniquely defined on a dense set, and by continuity on the whole space $\cH$. 
\end{proof}

\begin{proof}[Proof of Lemma \ref{l:uniq J_p}]
By Lemma \ref{l:Jp} any conjugation $\fJ_p$ commuting with $R$ and preserving $p$ is uniquely 
defined on $\cH_0$ (and is given there by \eqref{e: fJ_p 01}). 

From \eqref{e:K alt}, \eqref{e:K_1 alt} we can see that 
\begin{align*}
\cH =\cH_0 \bigoplus_{k} \ker (R_1 -\lambda_k I) \bigoplus_{k} \ker (R -\mu_k I) , 
\end{align*}
so we just need to define $\fJ_p$ on each of the eigenspaces. 

To define $\fJ_p$ on $\ker (R_1 -\lambda_k I)$ let us consider a bigger subspace
\begin{align*}
\ker (R -\lambda_k I) = \ker (R_1 -\lambda_k I) \oplus \spn\{k_p\},  
\end{align*}
and define $\fJ_{p,k}$ to be the unique conjugation there, preserving $p$ and such that  
$\bvf\big|_{\ker(R-\lambda_k I)}$ is $\fJ_{p,k}$-symmetric, see Lemma \ref{l:uniq Conj 01}. The 
fact that $\fJ_{p,k} p_k=p_k$ implies that $\ker(R-\lambda_k I)$ is a reducing subspace for 
$\fJ_{p,k}$, so the restriction $\fJ_{p,k}\big|_{\ker(R_1-\lambda_k I)}$ is well defined. 

Similarly, we define the conjugation $\fJ_{p,k}^1$ on $\ker(R_1-\mu_k I)$  to be the unique  
conjugation preserving $p_{k}^1$ and  such that $\bvf_1\big|_{\ker(R_1-\mu_k I)}$ is 
$\fJ_{p,k}^1$-symmetric. Again, the subspace $\ker(R-\mu_k I)$ is reducing for $\fJ_{p,k}^1$, so 
the conjugation   $\fJ_{p,k}^1\big|_{\ker(R-\mu_k I)}$ is well defined. 

Thus, taking for $\fJ_p\big|_{\cH_0}$ to be the unique conjugation on $\cH_0$ commuting with 
$R\big|_{\cH_0}$ and preserving $p$, and defining
\begin{align}
\label{e:restrictions of J_p}
\fJ_p\big|_{\ker(R_1-\lambda_k I)} := \fJ_{p,k}\big|_{\ker(R_1-\lambda_k I)}, \qquad
\fJ_p\big|_{\ker(R-\mu_k I)} := \fJ_{p,k}^1\big|_{\ker(R-\mu_k I)}, 
\end{align}
we get the conjugation $\fJ_p$ with the desired properties. 

Indeed, it preserves $p$ by the definition, and since $\bvf\big|_{\ker(R-\lambda_k I)}$ is 
$\fJ_{p,k}$-symmetric, we conclude that $\bvf\big|_{\cK}$ is $(\fJ_p \big|_{\cK})$-symmetric; here 
we used the fact that
\begin{align*}
\cK\oplus\cH_0 :=\cspn\{ \ker (R -\lambda_k I) : k\ge 1 \} , 
\end{align*}
see \eqref{e:K K_1}. Since $\bvf\big|_{\cK_1} =I\ci{\cK_1}$, we conclude that $\bvf$ is 
$\fJ_p$-symmetric. The fact that $\bvf_1$ is $\fJ_p$-symmetric is checked similarly, using the fact 
that $\bvf_1 \big|_{\ker(R_1-\mu_k I)}$ is $\fJ_{p,k}^1$-symmetric. 

As for the uniqueness of $\fJ_p$, we notice that for any such $\fJ_p$ the subspaces $\ker 
(R-\lambda_k I)$, $\ker (R_1-\mu_k I)$ are reducing for $\fJ_p$, and that by Lemma \ref{l:uniq Conj 
01} 
\begin{align*}
\fJ_p\big|_{\ker(R-\lambda_k I)} := \fJ_{p,k}, \qquad
\fJ_p\big|_{\ker(R_1-\mu_k I)} := \fJ_{p,k}^1, 
\end{align*}
with $\fJ_{p,k}$, $\fJ_{p,k}^1$ defined above. Thus the identities \eqref{e:restrictions of J_p} 
must hold, so the conjugation $\fJ_p$ is unique (recall, that as we already discussed above, 
$\fJ_p\big|_{\cH_0}$ is uniquely defined).  
\end{proof}

As we discussed above, the canonical choice of unitary operators $\bvf$ and $\bvf_1$ is given by 
the operators satisfying \eqref{e:canonical phi} and such that the vectors $p_k$ and $p_k^1$ are 
$*$-cyclic for $\bvf\big|_{\ker (R-\lambda_k I)}$ and $\bvf_1\big|_{\ker (R_1-\mu_k I)}$ 
respectively. 

Since $\cH=\cH_0\oplus \cK \oplus \cK_1$, and by \eqref{e:K K_1}
\begin{align*}
\cH_0\oplus \cK = \bigoplus_{k} \ker (R-\lambda_k I), 
\end{align*}
the restrictions $\bvf\big|_{\ker (R-\lambda_k I)}$ completely define the unitary operator $\bvf$, 
commuting with $R$, and such that $\bvf\big|_{\cK_1} =I\ci{\cK_1}$. 
Similarly, the restrictions $\bvf_1\big|_{\ker (R_1-\mu_k I)}$ completely define the unitary 
operator $\bvf_1$, commuting with $R_1$, and such that $\bvf_1\big|_{\cK} =I\ci{\cK}$. 

Finally, we want to get unitary invariant description of the spectral data. 
If we want to define the triple $\bvf$, $\bvf_1$, $p$ up to unitary equivalence, we just 
need to define each pair $\bvf\big|_{\ker (R-\lambda_k I)}$, $p_k$ and $\bvf_1\big|_{\ker 
(R_1-\mu_k I)}$, $p^1_k$ up to (separate) unitary equivalence.  And since we already know the 
triple $R$, $R_1$, $p$, that allows us to get the operators $\bvf$ and $\bvf_1$.   

\subsection{The canonical spectral data for compact operators}
\label{s:CASD compact}
Summarizing results of the previous subsection, let us give a simpler description of the abstract 
spectral data in the compact case. Let us consider only the abstract spectral data that could
correspond to a Hankel operator, i.e.~let us ignore the data for which we know for sure that the 
operator $\Sigma^*= \bvf_1 R_1 R^{-1}\bvf^*$ is not asymptotically stable. 

As it was discussed before, such abstract spectral data is given by a tuple $R$, $R_1$, $p$, 
$\bvf$, $\bvf_1$, where, in the notation of Section \ref{s:ASD compact} 
\begin{enumerate}
\item $R\ge0$,  $R_1\ge 0$, are self-adjoint compact operators, $\ker R=\{0\}$,  satisfying $R_1^2 
= R^2 
-pp^*$, $p\in\cH$. We also require that $\lambda_k$, $\mu_k$ are the only possible eigenvalues of 
$R$,  or, 
equivalently $\cH= \cH_0\oplus\cK\oplus\cK_1$, see Lemma \ref{l:spectrum R}
\item $\bvf$ and $\bvf_1$ are unitary operators, commuting with $R$ and $R_1$ respectively, 
satisfying  \eqref{e:canonical phi}. 
We also assume that the vectors  $p_k=P\ci{\ker (R-\lambda_k 
I)}p$, and $p^1_k:= P\ci{\ker (R_1-\mu_k I)}p$ are $*$-cyclic for $\bvf\big|_{\ker (R-\lambda_k 
I)}$ and $\bvf_1\big|_{\ker (R_1-\mu_k I)}$ respectively, see Lemma \ref{l:cyclicity 01}.  
\end{enumerate}

The conjugation $\fJ_p$ is uniquely defined by the above abstract spectral data.

\subsection{A simple representation of the abstract spectral data}
Let $\bu$ be a unitary operator, $p$ be a $*$-cyclic vector for $\bu$, and $\rho$ be a spectral 
measure of $\bu$ corresponding to the vector $p$.  Recall that by the spectral theorem the pair 
$\bu$, $p$ is unitarily equivalent to the pair $M^\rho$, $\1\in L^2(\rho)$, where $M^\rho$ is the 
multiplication operator by the independent variable in $L^2(\rho)$: $M^\rho f(z) =z f(z)$, $f\in 
L^2(\rho)$. 

Note also that $\|\rho\| =\var \rho = \|p\|^2$. 

Recall also that  that if $\bu$ has  finite rank, then the spectral measure $\rho$ is \emph{finitely
supported} (i.e.~its support consists of finitely many point, or equivalently, it is represented as
$\rho=\sum_{k=1}^m a_k\delta_{\xi_k}$, $n=\rank \bu$).

To summarize the result of the previous subsection, the abstract spectral data in the case of  
compact gives us 

\begin{enumerate}
\item Two sequences of positive numbers $\{\lambda_k\}_{k\ge1}$ and $\{\mu_k\}_{k\ge1}$ 
(eigenvalues of $R\big|_{\cH_0}$ and $R_1\big|_{\cH_0}$), satisfying 
the intertwining relations \eqref{e:Intertwine lambda mu}. Here and $\{\lambda_k\}_{k\ge1}$ is 
either a finite sequence $\lambda_1, \lambda_1, \ldots, \lambda_n$ (in this case 
$\mu_n$ can be $0$), or $\lambda_k \searrow 0$. 

\item Finitely supported probability measures $\rho_k$ and $\rho^1_k$, where $\rho_k$ is the
spectral measure  of the operator  $\bvf\big|_{\ker (R-\lambda_k I)}$ with respect to the unit
vector $\|p_k\|^{-1}p_k$, and $\rho^1_k$ is the  spectral measure  of the operator 
$\bvf_1\big|_{\ker (R-\mu_k I)}$ with respect to the unit vector $\|p^1_k\|^{-1}p^1_k$.

\end{enumerate}

We claim that  if we are given sequences  $\{\lambda_k\}_{k\ge1}$,  $\{\mu_k\}_{k\ge1}$ of positive 
numbers, and two sequences $\{\rho_k\}_{k\ge1}$,  $\{\rho^1_k\}_{k\ge1}$ of finitely supported 
probability measures, we can find the corresponding abstract spectral data, i.e.~the tuple $R$, 
$R_1$, $p$, $\fJ_p$, $\bvf$, $\bvf_1$ (defined up to unitary equivalence, with $\bvf$ and $\bvf_1$ 
canonically chosen to satisfy \eqref{e:canonical phi}), such that  the above four sequences 
correspond to this abstract spectral data. 

First, by the Abstract Borg's Theorem (Theorem \ref{t:Borg 01}), the sequences 
$\{\lambda_k\}_{k\ge1}$ and $\{\mu_k\}_{k\ge1}$ define the triple $R\big|_{\cH_0}$, 
$R_1\big|_{\cH_0}$, $p$, up to unitary equivalence. 

That means we know the spectral measures $\rho$ and $\rho^{[1]}$, and so the weights $w_k$ and 
$w^{[1]}_k$ from \eqref{e:repr rho compact case}. Note than this also give us the norms $\|p_k\|$, 
$\|p^1_k\|$, see \eqref{e:norm p_k}, \eqref{e:norm p^1_k}. 

The dimensions of eigenspaces must be given by the cardinalities of the supports of the measures 
$\rho_k$ and $\rho^1_k$, 
\begin{align*}
\dim \ker (R-\lambda_k I )  &=\card \supp \rho_k, &\dim \ker (R_1-\mu_k I )    &=\card \supp 
\rho^1_k \\
\dim \ker (R-\mu_k I )    &=\card \supp \rho^1_k-1,  & \dim \ker (R_1-\lambda_k I )  &=\card \supp 
\rho_k  - 1.    
\end{align*}
Thus, the abstract spectral defines the triple $R$, $R_1$, $p$ up to unitary equivalence. 

Namely, we can define subspaces $E_{\lambda_k}, E_{\mu_k}\subset \ell^2 = \ell^2(\Z_+)$ as 
\begin{align*}
E_{\lambda_k} & :=\spn\{ e_j : 0\le j \le \card \supp \rho_k -1 \}, \\
E_{\mu_k} & :=\spn\{ e_j : 1\le j \le \card \supp \rho^{[1]}_k -1 \}, 
\end{align*}
and define  $\cH$ as the direct sum 
\begin{align*}
\cH=\Bigl(\bigoplus_{k\ge1} E_{\lambda_k}\Bigr)\oplus \Bigl(\bigoplus_{k\ge1} E_{\mu_k}\Bigr). 
\end{align*}
Defining $p_k\in E_{\lambda_k}$ by $p_k:= (w_k)^{1/2} e_0$, and $p:=\oplus_{k\ge1} p_k$, we get the 
triple $R$, $R_1$, $p$ up to unitary equivalence. In this representation 
$\cH_0=\cspn\{p_k:k\ge1\}$, the operator $R_1\ge0$ is defined from the rank one perturbation 
relation $R_1^2 = R^2 - pp^*$. The vectors $p^1_k=P\ci{\ker(R_1-\mu_k I)}\in \cH_0$ can be computed 
from this rank one perturbation relation. 

The eigenspaces of $R$ are given by 
\begin{align*}
\ker (R-\lambda_k I) &= E_{\lambda_k}, & \ker (R-\mu_k I) &= E_{\mu_k} , 
\intertext{and for $R_1$, see \eqref{e:defect R_1}, \eqref{e:defect R_1}, by}
\ker (R_1-\lambda_k I) &= E_{\lambda_k}\ominus \spn\{p_k\},  &\ker (R-\mu_k I) &= E_{\mu_k} \oplus 
\spn\{p^1_k\},  
\end{align*}

Finally, to define the canonical unitary operators $\bvf$, $\bvf_1$ satisfying \eqref{e:canonical 
phi}, we need to construct the restrictions $\bvf\big|_{\ker (R-\lambda_k I)}$ and 
$\bvf_1\big|_{\ker (R_1-\mu_k I)}$. 

The spectral measure $\rho_k$ defines the pair $\bvf\big|_{\ker (R-\lambda_k I)}$, $\|p_k\|^{-1} 
p_k$ up to unitary equivalence, and since we already know $\|p_k\|=(w_k)^{1/2}$, the pair 
$\bvf\big|_{\ker (R-\lambda_k I)}$, $p_k$ is also defined up to unitary equivalence. 
Similarly, the spectral measure $\rho^{[1]}_k$ defines the pair $\bvf_1\big|_{\ker (R_1-\mu_k I)}$, 
$p^{[1]}_k$ up to unitary equivalence. 
And as it was discussed at the end of section \ref{s:canonical J_p}, that is enough to define the 
operators   $\bvf$, $\bvf_1$.

\section{The inverse problem for compact Hankel operators} 
\label{asymptotic stability}

To check if the operator $\Sigma^*$ is asymptotically stable, is usually a very hard problem. 

However, in the case of compact operator $R$, if we exclude obvious obstacles, and consider the 
abstract spectral data as defined in Section \ref{s:CASD compact}, we will get the asymptotic 
stability essentially  for free. 

\begin{defin}
\label{d:WAS}
We say that an operator $A$ is \emph{weakly asymptotically stable}, if $A^n\to 0$ in the weak 
operator topology (W.O.T) as  $n\to\infty$. 
\end{defin}

\begin{lm}
\label{l:WAS A to AS T}
 Let $\| T \|\le 1$, and let $K$ be a compact operator with dense range. 
 Assume that  an operator $A$ satisfies
 \begin{align}
 \label{TK=KA}
T K = K A. 
 \end{align}
If $A$ is weakly asymptotically stable,  then $T$ is asymptotically stable. 
\end{lm}
The proof is elementary, an we leave it as an exercise for the reader; see also \cite{Treil Liang 2022}.  

In this section we will construct an operator $A$ satisfying 
\begin{align}
\label{e:prop of A}
\Sigma^* R^{1/2} = R^{1/2} A. 
\end{align}
We will show that for the abstract spectral data from Section \ref{s:CASD compact} the operator $A$ 
is weakly asymptotically stable; since for such spectral data $R$ (and so $R^{1/2}$) is compact, we 
immediately conclude that $\Sigma^*$ is asymptotically stable. 

All the steps are pretty elementary, so essentially we will get the asymptotic stability for free. 

\subsection{Construction of \tp{$A$}{A}}
\label{s:A constr}  One can immediately see that an operator $A$ formally 
given by 
\begin{align*}
A=R^{-1/2} \Sigma^* R^{1/2} = R^{-1/2}  R_1\bvf_1 \bvf^* R^{1/2}
\end{align*}
formally satisfies \eqref{e:prop of A}. Let us make this construction rigorous. 

We know that $R_1^2\le R^2$, so by Heinz inequality with exponent $1/2$ we have that $R_1\le R$, 
and so by Douglas Lemma (Lemma \ref{l:Douglas}), the operator $Q:=R_1^{1/2}R^{-1/2}$ (defined 
initially on a dense set $\ran R^{1/2}$) extends to a contraction ($\|R_1^{1/2} R^{-1/2}\|\le 1$), 
and its adjoint is given by $Q^*=R^{-1/2}R^{1/2}$ (and is defined on the whole space $\cH$). 

Defining 
\begin{align}
\label{e:def A}
A:= Q^*\bvf_1 Q\bvf^*
\end{align}
we immediately see that $\|A\|\le 1$. Using the commutation 
relations $\bvf_1R_1^{1/2} = R_1^{1/2} \bvf_1$ and $\bvf^* R^{-1/2} = R^{-1/2}\bvf^*$, for the 
latter see Remark \ref{r:commutation}, we can write $A= R^{-1/2}  R_1\bvf_1 \bvf^* R^{1/2}$ and see 
that \eqref{e:prop of A} is satisfied. 

\subsection{The structure of \tp{$Q$}{Q}}
Recall that we defined $\cH_0$ as $\cH_0:=\cspn\{R^n p: n\ge 0\}$. 

\begin{lm}
\label{l:structure of Q}
The operator $Q$ with respect to the decomposition $\cH = \cH_0 \oplus \cH_0^\perp$ has the  
block diagonal structure 
\begin{align}
\label{e:struct Q}
Q=\begin{pmatrix}
Q_{0} & 0 \\
0 &  I  \end{pmatrix}, 
\end{align}  
where $Q_0$ is a strict contraction \tup{(}i.e.~$\|Q_0 x\|<\|x\|$ for all $x\ne 0$\tup{)}. 
\end{lm}
The proof can be found in \cite{Treil Liang 2022}. For the reader's convenience we 
present it here. 
\begin{proof}[Proof of Lemma \ref{l:structure of Q}]
First, notice that $\cH_0$ and $\cH_0^\perp$ are invariant for both $R^2$ and $R_1^2$ and that 
$R^2$ and $R_1^2$ coincide on $\cH_0^\perp$. Therefore $\cH_0$ and $\cH_0^\perp$ are both invariant 
for $R^{1/2}$ and $R_1^{1/2}$ and $R^{1/2}$ coincides with $R_1^{1/2}$ on $\cH_0^\perp$, so the 
operator $Q$ has the block diagonal structure \eqref{e:struct Q}. Thus, to prove the lemma, we only 
need to show that $Q_0$ is a strict contraction.

Direct computations show that 
\begin{align}
\notag
Q^*QRQ^*Q & = R^{-1/2} R_1^2 R^{-1/2} = R^{-1/2} (R^2 - pp^*) R^{-1/2} 
\\  \label{e:Q*QRQ*Q}
&= R- R^{-1/2}p (R^{-1/2}p)^* 
\end{align}
(note that since $\|R^{-1} p\|\le 1$ we can conclude that $\|R^{-1/2} p\|^2 = \la R^{-1}p,p\ra \le 
\|p\|$, so $R^{-1/2}p \in\cH$). 

Take now $x\in\cH$ such that $\|Qx\|=\|x\|$. This is equivalent to the identity 
\begin{align*}
\la (I-Q^*Q)x, x\ra = 0, 
\end{align*}
and since $I-Q^*Q\ge 0$, is equivalent to $(I-Q^*Q)x=0$, or equivalently, that $Q^*Qx=x$. 

Using identity $Q^*Qx=x$ and \eqref{e:Q*QRQ*Q} we get 
\begin{align*}
\la Rx,x\ra=\la Q^*QRQ^*Q x,x\ra = \la Rx,x\ra -|\la x, R^{-1/2} p\ra|^2,  
\end{align*}
so $x\perp R^{-1/2}p$. Then 
\begin{align*}
Q^*QR x= Q^*QRQ^*Q x = Rx - R^{-1/2}p (R^{-1/2}p)^* x = Rx;
\end{align*}
the first equality here follows from $Q^*Qx=x$, the second one from \eqref{e:Q*QRQ*Q}, and the last 
one from the above orthogonality $x\perp R^{-1/2}p$. 

Therefore $Q^*QRx=Rx$, so $\|Q^*QRx\|=\|Rx\|$. Thus
\begin{align*}
\cH_1:=\{x\in\cH: \|Qx\|=\|x\|\} = \ker (I-Q^*Q)
\end{align*}
is an $R$-invariant subspace, orthogonal to $R^{-1/2}p$. 

Since $\cH_1$ is $R$-invariant, it is also $R^{1/2}$-invariant, so $\cH_1$ is orthogonal to 
\begin{align*}
\cspn\{ R^{-1/2 + k/2} : k\ge 0\} \subset \cspn\{ R^{ k/2} : k\ge 0\} \supset \cH_0
\end{align*}
(in fact, we have equalities above, not just inclusions, but this is not necessary for the proof), 
so $\|Qx\|<\|x\|$ for any $x\in \cH_0\setminus\{0\}$. But this exactly means that 
$Q_0=Q\big|_{\cH_0}$ is a 
strict contraction. 
\end{proof}

\subsection{Weak asymptotic stability of a completely non-unitary contraction}
\label{s:CNU is WAS}
Recall that a contraction $T$ is called \emph{completely non-unitary}, if there is no reducing 
subspace for $T$ on which $T$ acts unitarily. 

\begin{lm}
\label{l:cnu implies weakly stability}
A completely non-unitary contraction $T$  on a Hilbert space $\cH$ is always weakly 
asymptotically stable.
\end{lm}

We will show later in Section \ref{s:A is CNU} that the contraction $A$ constructed above in 
Section \ref{s:A constr} is weakly asymptotically stable, which will complete the proof of 
asymptotic stability of $\Sigma^*$. 

\begin{proof}[Proof of Lemma \ref{l:cnu implies weakly stability}]
Every completely non-unitary contraction $T$ admits the  \emph{functional mod\-el}, i.e. it 
is
unitaryly equivalent to the model operator $\cM_\theta$ on the model space $\cK_\theta$, where
$\theta$ is the so-called \emph{characteristic function} of $T$, see for example \cite[Sect.\ 
VI.2]{Harmonic analysis 2010}, \cite[Ch.~1, Sect.~1.3]{Nikolski2}. Without going into details, which
are not important for our purposes, we just mention that the model space $\cK_\theta$ is a subspace
of a vector-valued space $L^2(E)=L^2(\T,  m;E)$ of square integrable (with respect to the normalized
Lebesgue measure $m$ on $\T$) functions with values in an auxiliary Hilbert space $E$. The model
operator $\cM_\theta$, to which $A$ is unitarily equivalent,  is just the \emph{compression} of the
multiplication operator$M_z$ by the independent variable $z$
\begin{align*}
\cM_\theta f = P\ci{\cK_\theta} M_z f, \qquad f\in \cK_\theta ;
\end{align*}
recall that  the  multiplication operator $M_z$ is defined by $M_z f(z) = zf(z)$, $z\in\T$. 

What is also essential for our purposes, is that the multiplication operator $M_z$ is the 
\emph{dilation} of the model operator $\cM_\theta$, i.e.~that for all $n\ge 1$
\begin{align*}
\cM_\theta^n f = P\ci{\cK_\theta}M_z^n f, \qquad f\in \cK_\theta. 
\end{align*}

Since trivially $M_z^n \to 0$ in the weak operator topology of $B(L^2(\T, m;E))$ as $n\to+\infty$, 
we conclude that $\cM_\theta^n\to 0$ as $n\to +\infty$ in  the weak operator topology of 
$B(\cK_\theta)$, and so $T^n\to 0$ in the weak operator topology as well. 
\end{proof}

\subsection{\tp{$A$}{A} is completely non-unitary}
\label{s:A is CNU}

\begin{prop}
\label{p:contraction is cnu}
The contraction $A$ defined above in Section\ref{s:A constr} is completely non-unitary if and only 
if $\bvf_1 \bvf^*$ does not have 
any non-zero reducing subspace  $E\subset \cH_0^{\perp}$ such that $\bvf^* E\perp \cH_0$.
\end{prop}

To prove this proposition we need the following simple observation. 
\begin{lm}
\label{equivalent condition of cnu}
Let $T$ be a contraction and $E$ be a subspace. The following statements are equivalent:
\begin{enumerate}
	\item The subspace $E$ is a reducing subspace for $T$ such that $T\big|_{E}$ is unitary;
	\item The operator $T$ acts isometrically on $E$ (i.e.~$\|Tx\|=\|x\|$ for all $x\in E$) and 
	$TE=E$. 
\end{enumerate}
\end{lm}
The proof of this lemma is trivial, and we omit it. 

\begin{lm}
\label{l:T^* strict}
If $T$ is a strict contraction, then $T^*$ is also a strict contraction. 
\end{lm}
\begin{proof}
If $T$ is a strict contraction, then $|T|$ is also a strict contraction. Using the polar 
decomposition $T=U|T|$, we see that $T^* = |T|U^*$ is also a strict contraction. 
\end{proof}

\begin{proof}[Proof of Proposition \ref{p:contraction is cnu}]
Assume that 
$E\subset \cH_0^\perp$ is a reducing   subspace for $\bvf_1\bvf^*$ such that $\bvf^* E 
\subset \cH_0^\perp$.

Since $Q =Q^*= I$ on $\cH_0^{\perp}$, we have for any $x \in E$, 
\begin{align}
\label{e:Ax=phi_1 phi* x}
A x = Q^* \bvf_1 Q \bvf^* x = Q^* \bvf_1 \bvf^* x = \bvf_1 \bvf^* x;
\end{align}
in the second equality we used the fact that $\bvf^* x \in \cH_0^\perp$, and in the last one the 
fact that $\bvf_1\bvf^* x \in E\subset \cH_0^\perp$. 

Since $\bvf_1$ acts isometrically on $\cH_0^{\perp}$ ($\Ker \bvf_1$ can only belongs to 
$\cH_0$), we have
\begin{align*}
\| A x \| = \| \bvf_1 \bvf^* x \| = \| x \|,
\end{align*}
i.e. $A$ acts isometrically on $E$. 

Also, this implies that $\bvf_1\bvf^*$ acts isometrically on $E$. 
In addition, $\bvf$ is unitary, and (see Remark \ref{varphi 1 is partial unitary}) we know 
$\bvf_1\big|_{(\Ker R_1)^{\perp}} $  is unitary, therefore $\bvf\bvf_1^*$ acts unitarily 
on its reducing subspace $E$. Now from \eqref{e:Ax=phi_1 phi* x} we have
\begin{align*}
AE = \bvf_1 \bvf^* E = E.
\end{align*}
Applying Lemma \ref{equivalent condition of cnu},  we can see that $E$ is a reducing subspace for 
$A$ such that $A\big|_{E}$ is unitary, so $A$ is not completely non-unitary.

Now let us prove the opposite implication. If $A$ is not completely non-unitary, then we can find a 
reducing subspace 
$E$ for $A$, such that $A\big|_{E} = Q^* \bvf_1 Q \bvf^* \big|_{E}$ is unitary. Using the fact that 
$Q$
is a pure contraction on $\cH_0$, and that $\|A^*x\|=\|\bvf Q^* \bvf_1^* Q x\|\le \|Qx\|$, we 
conclude that $E\perp\cH_0$. Similarly, since $\|Ax\|=\|Q^*\bvf_1 Q\bvf^*x\|\le \|Q\bvf^*x\|$, 
\begin{align*}
&\bvf^* E \perp \cH_0. 
\end{align*}
Since $Q\big|_{\cH_0^\perp} = I\ci{\cH_0^\perp}$, we have for $x\in E$
\begin{align*}
Ax = Q^*\bvf_1 Q\bvf^* x = Q^* \bvf_1 \bvf^* x,   
\end{align*}
By Lemma \ref{l:T^* strict} the operator $Q^*\big|_{\cH_0} = Q_0^*$ is a strict contraction, and 
since 
$A\big|_E$ is unitary, we see that $\bvf_1\bvf^* E \perp \cH_0$, and that 
\begin{align*}
A\big|_E = \bvf_1\bvf^* \big|_E. 
\end{align*}
Therefore $E$ is a reducing subspace for $\bvf_1\bvf^*$.  
\end{proof}

To complete the proof that $A$ is completely non-unitary, we will show that that there is no 
non-trivial reducing subspace $E\perp\cH_0$ of $\bvf_1\bvf^*$. 

Let $E$ be a reducing subspace of $\bvf_1\bvf^*$, $E\perp\cH_0$, and let $f\in\cH$. Then 
\begin{align}
\label{e:non-cyclic 01}
(\bvf_1\bvf^*)^n f \perp \cH_0 \qquad \forall n\in \Z. 
\end{align}
Let us decompose $f$ as 
\begin{align*}
f &=\sum_k f_{\lambda_k} + \sum_k f_{\mu_k}, \\ 
f_{\lambda_k} &= P\ci{\ker(R-\lambda_k I)}f = P\ci{\ker(R_1-\lambda_k I)}f, \\
f_{\mu_k} &= P\ci{\ker(R_1-\mu_k I)}f = P\ci{\ker(R-\mu_k I)}f.
\end{align*}
Taking into account \eqref{e:canonical phi} we can see that the  condition \eqref{e:non-cyclic 01} 
is equivalent to the fact that for all $k$ and for all $n\in Z$
\begin{align*}
\la \bvf^n f_{\lambda_k} , p_k\ra = \la  f_{\lambda_k} , \bvf^{-n} p_k\ra =0,   \quad
\la \bvf_1^n f_{\mu_k} , p^1_k\ra = \la  f_{\mu_k} , \bvf_1^{-n} p^1_k\ra =0\qquad 
\forall n\in \Z ,
\end{align*}
which is impossible because vectors $p_k$ and $p_k^1$ are $*$-cyclic for 
$\bvf\big|_{\ker(R-\lambda_k I)}$ and $\bvf_1\big|_{\ker(R_1-\mu_k I)}$ respectively. \hfill\qed

\section{Relation to previous work}
\label{s:previous work}

\subsection{Quick recap of Clark theory}
\label{s:Clark recap}

Let us quickly describe basic facts from Clark theory, \cite{Clark}, that we will use. 

Recall that a function $\theta\in H^2$ is called \emph{inner}, if $|\theta|=1$ a.e.~on $\T$. 

For an inner function $\theta$,  the  \emph{model space} $\cK_\theta$ is defined as  $\cK_\theta := 
H^2\ominus\theta H^2$, and the \emph{model operator} $\cM_\theta:\cK_\theta \to\cK_\theta$ is the 
compression of the forward shift
\begin{align*}
\cM_\theta f := P\ci{\cK_\theta} Sf, \qquad f\in \cK_\theta. 
\end{align*}

We consider only the simplest version of the theory when $\theta$ is an inner function satisfying 
$\theta(0)=0$, although in \cite{Clark} the case of general inner functions was treated. 

If $\theta(0)=0$, then $z^0, \theta\in \cK_\theta$, and one can consider a rank one unitary 
perturbation $\cU=\cU_\theta$ of $\cM_\theta$ (i.e. a unitary operator $\cU$ such that 
$\rank(\cM_\theta-\cU)=1$), 
\begin{align*}
\cU f = \left\{
\begin{array}{ll} \cM_\theta f = Sf \qquad & f\in \cK_\theta \ominus \theta; \\
				   z^0, & f=\theta	
\end{array}\right.
\end{align*}

The Clark measure $\rho=\rho_\theta$ in this case is the spectral measure of the operator 
$\cU_\theta$ with 
respect to the 
vector $z^0$. Alternatively, it can be defined as the Borel measure such that 
\begin{align*}
\re \frac{1+\theta}{1-\theta} = \text{Poisson extension of }\rho, 
\end{align*}
in the unit disc $\D$, or, equivalently, 
\begin{align*}
\frac{1+\theta(z)}{1-\theta(z)} = \int_\T \frac{1+z\overline \xi}{1-z\overline \xi} \dd\rho(\xi), 
\qquad z\in \D. 
\end{align*}
For an inner function the Clark measure is always  singular, and if $\theta(0) =0$, then $\rho$ is 
a probability measure.  Also, if $\theta$ is a finite Blaschke product, then $\rho$ is a measure
with finite support 
\begin{align*}
\rho =\sum_{k=1}^n a_k \delta_{\xi_k}, \qquad n = \deg \theta. 
\end{align*}

The Clark operator $\cC=\cC_\theta : \cK_\theta \to L^2(\rho)$, $\rho=\rho_\theta$ is a unitary 
operator, intertwining $\cU_\theta$ and the multiplication operator $M_\rho: L^2(\rho)\to
L^2(\rho)$, 
\begin{align*}
\cC \cU_\theta = M_\rho \cC, 
\end{align*}
is given by the formula 
\begin{align}
\label{e:Clark 01}
\cC f(\xi) = \lim_{r\to 1^-} f(r\xi), \qquad f\in\cK_\theta. 
\end{align}
It was shown in \cite{Clark} that convergence in \eqref{e:Clark 01} can be understood in the sense 
of convergence in $L^2(\rho)$, and it was later shown in \cite{Polt93} that convergence in 
\eqref{e:Clark 01} is $\rho$-a.e., and that this even holds for non-tangential boundary values.  
However, the reader does not need to worry about it, because in the case that we are interested in, 
when $\theta$ is a finite Blaschke product, the subspace $\cK_\theta$ consists of rational 
functions,  so there is no question about convergence. 

Note also that $\cC z^0 =\1\in L^2(\rho)$.

\subsection{Description of previous results}
To describe the previous results we will treat the Hankel operator as an operator in the Hardy space 
$H^2$. As it was mentioned in the Introduction,  $z$-transform gives the unitary equivalence 
between representation in $\ell^2$ and $H^2$; so in this section we will treat Haknel operators as 
operators in $H^2$, but will keep the same notation. 

In \cite{Gerard 2014 paper 2}, the authors considered the conjugate-linear Hankel operator $H_u$, $u\in H^2$, 
\begin{align*}
H_u (f) = P_+ (u\bar f), \qquad f\in H^2;
\end{align*}
where $P_+$ is the Szeg\"{o} projection, i.e.~the orthogonal projection from $L^2 = L^2(\T)$ onto 
$H^2$. Note that $H_u z^0 = u$. 

Note that in our notation $H_u = \fC\Gamma  = \Gamma^* \fC$. In our notation we defined $u=\Gamma^* 
e_0$, so vectors $u$ in both cases agree. 

They also introduced the operator $K_u=H\ci{S^*u} = S^* H_u = \fC\Gamma_1 = \Gamma_1 \fC^*$. We 
should note here that $H_u^2 = \Gamma^*\Gamma$ and $K_u^2 = \Gamma_1^*\Gamma_1$.  

The following result was  obtained in \cite{Gerard 2014 paper 2, GGAst, GerPush2020} for the case 
of compact operators. 
Below we will follow \cite{GerPush2020}. 

Let a compact Hankel operator $\Gamma$ be fixed, and let $H_u=\fC\Gamma$, $K_u = \fC\Gamma_1$ be 
the corresponding 
conjugate linear Hankel operators. For $s>0$ denote 
\begin{align*}
E(s) := \ker (|\Gamma| - s I) = \ker (H_u^2 -s^2 I), \quad E_1(s)  := \ker (|\Gamma_1| - s I) 
= \ker (K_u^2 -s^2 I);
\end{align*}
in \cite{GerPush2020} the notation $H_u(s)$, $K_u(s)$ was used. 

It is well known, see \cite[Lemma 2.1]{GerPush2020}, that if $s$ is an eigenvalue of either 
$|\Gamma|$ or $|\Gamma_1$, then one (and only one) of the following properties holds:

\begin{enumerate}
\item $u\not\perp E(s)$ and $E_1 (s) = E(s)\cap u^\perp$; 
\item $u\not\perp E_1(s)$ and $E (s) = E_1(s)\cap u^\perp$.  
\end{enumerate}

\begin{thm}
\label{t:H_u rf}
Let $s$ be an eigenvalue of either $|\Gamma|$ or $|\Gamma_1|$. 
\begin{enumerate}
\item Let $u\not\perp E(s)$; denote by $r_s$ and $u_s$ the orthogonal projection of $z^0$ 
(respectively of $u$) onto $E(s)$. Then $u_s(z) = s \bar z \theta_s(z) r_s(z)$ for an inner 
function $\theta_s$, $\theta_s(0)=0$.  The normalized function $\|r_s\|^{-1} r_s$ is an 
isometric multiplier from $\cK_{\theta_s} $ to $ r_s \cK_{\theta_s} $, and 
\begin{align}\label{e:E(s)} 
E(s) = r_s \cK_{\theta_s}. 
\end{align}
The action of 
$H_u$ on $E(s)$ is given by 
\begin{align}
\label{e:H_u uf}
\left(H_u r_s f\right) (z) = s  r_s(z) \theta_s(z)\overline z \overline{ f(z)}   , \qquad \forall 
f\in \cK_{\theta_s},   
\end{align}
and the inner function $\theta_s$ is uniquely defined by \eqref{e:E(s)} and \eqref{e:H_u uf}. 

\item Let $u\not\perp E_1(s)$; denote by $u^{1}_s$  the orthogonal projection of $u$ 
 onto $E_1(s)$. Then $\left(K_u u^{1}_s\right)(z) = s \bar z \theta^1_s(z) u^{1}_s(z)$ for an inner 
function $\theta^{1}_s$, $\theta^{1}_s(0)=0$.  The normalized function $\|u^{1}_s\|^{-1} u^{1}_s$ 
is an isometric multiplier from $\cK_{\theta^{1}_s} $ to $ u^{1}_s \cK_{\theta^{1}_s} $ and 
\begin{align}\label{e:E_1(s)} 
E_1(s) = u^{1}_s \cK_{\theta^{1}_s}. 
\end{align}
The action of 
$K_u$ on $E_1(s)$ is given by 
\begin{align}
\label{e:K_u uf}
\left(K_u u^{1}_s f\right) (z) = s  u^{1}_s(z) \theta^{1}_s(z)\overline z \overline{ f(z)}   , 
\qquad \forall 
f\in \cK_{\theta_s},   
\end{align}
and the inner function $\theta^{1}_s$ is uniquely defined by \eqref{e:E_1(s)} and \eqref{e:K_u uf}. 
\end{enumerate}
\end{thm}

\begin{rem}
\label{r:H_u rf notation}  
Few observations:
\begin{enumerate}
\item Since $\Gamma$ is a compact operator, the inner functions $\theta_s$ and $\theta^{1}_s$ are 
finite  
Blaschke products, $\deg \theta_s = \dim E (s)$, $\deg \theta^{1}(s) = \dim E_1(s)$. 

\item In \cite{GerPush2020} a slightly different notation was used. In particular, their inner 
functions 
$\psi_s$, $\wt\psi_s$, are related to our inner functions by $\theta_s(z) = z \psi_s(z)$, 
$\theta^{1}_s(z) = z\wt\psi_s$. 
\end{enumerate}
\end{rem}


Let us describe the spectral data presented in \cite{Gerard 2014 paper 2, GGAst}. Let us denote by 
$\cH_0:=\cspn\{|\Gamma|^n u: n\ge 0\}$. The spectral data then consisted of two sequences 
$\{\lambda_k\}_{k\ge 1}$, $\{\mu_k\}_{k\ge 1}$ corresponding to eigenvalues of 
$|\Gamma|\big|_{\cH_0}$ and to non-zero eigenvalues of $|\Gamma_1|\big|_{\cH_0}$ respectively, and 
two sequences of finite Blaschke products $\{\theta_k=\theta_{\lambda_k}\}_{k\ge1}$, 
$\{\theta^{1}_k=\theta^{1}_{\lambda_k}\}_{k\ge 1}$, 
$\theta_k (0)=0$, $\theta^{[1]}_k(0)=0$, where  

 $\theta_{\lambda_k}$ and $\theta^{1}_{\mu_k}$ are the  unique inner functions from the above 
 Theorem \ref{t:H_u rf} corresponding to the eigenvalues $s=\lambda_k$ and $s=\mu_k$ respectively.

\subsection{Connection with our results}
Now, let $s=\lambda_k$, $\theta=\theta_s=\theta_{\lambda_k}$ be the inner function from Theorem 
\ref{t:H_u 
rf}, part \cond1, and let $\rho=\rho_s=\rho_{\lambda_k}$ be the Clark measure for $\theta$, as 
defined 
above. 

Then from the description \eqref{e:Clark 01} of the Clark operator $\cC$, taking into account that 
$\theta(z)=1$ on $\supp \rho$, we can see, that
\begin{align*}
\cC \left(\theta(z)\bar z \overline{f(z)} \right)(\xi) = \bar\xi \overline{\cC f(\xi)}\qquad \xi\in 
\supp\rho \subset \T. 
\end{align*}

Theorem \ref{t:H_u rf} and the unitarity of the Clark operator $\cC$ imply  that $U = U_s : 
L^2(\rho)\to \Ker (|\Gamma|-s I)$, 
\begin{align}
\label{e:U Clark}
Uf = \|r_s\|^{-1} r_s \cC^* f , \qquad f\in L^2(\rho)
\end{align}
is unitary, and that 
\begin{align}
\label{e:model prelim}
\left(U^* H_u \big|_{E(s) } U f\right)(\xi) = s \bar\xi \overline 
{f(\xi)} .
\end{align}
Note also that 
\begin{align}\label{e:U* u_s}
\|u_s\|^{-1}\left(U^*u_s\right) (\xi) = \cC\left( z\mapsto  \bar z \theta_s(z) \right)(\xi) = 
\bar\xi \qquad \rho_{\lambda_k}\text{-a.e.}; 
\end{align}
here we used the fact that $\|u_s\|= s\|r_s\|$ and that $\theta_{\lambda_k} = 1$ 
$\rho_{\lambda_k}$-a.e.

Let now $\fc$ be the complex conjugation on $\T$, $\fc(\xi)=\bar \xi$, and let $\rho^\sharp := 
\rho\circ \fc$ be the reflection of the measure $\rho$.  Define the unitary operator $\cU=\cU_s: 
L^2(\rho^\sharp) \to \ker(|\Gamma|- s I) $ by 
\begin{align*}
\left(\cU^* f\right)(\xi) = \bar\xi\,\left(U^* f\right)(\bar\xi)  = \bar\xi \,\bigl[\cC  
(\|r_s\| f/r_s)\bigr](\bar \xi), \qquad 
f\in 
r \cK_{\theta_{s}}, 
\end{align*}
where $U$ is the unitary operator defined by \eqref{e:U Clark}. Then we can conclude from 
\eqref{e:model prelim} that 
\begin{align}
\label{e:model final}
\left(\cU^* \fC\Gamma \big|_{E(s) } \cU f\right)(\xi) = \xi \widebar 
{f(\xi)} , 
\end{align}
and that 
\begin{align*}
\|u_s\|^{-1}\cU^* u_s =\1. 
\end{align*}

One can see from Lemmas \ref{l:uniq J_p}, \ref{l:uniq Conj 01} and Remark \ref{r:J in spectral 
repr}  that 
\begin{align*}
\cU^* \fJ_u \big|_{E(s) } \cU f =  \widebar {f} , 
\end{align*}
and this operator is the representation of the conjugation $\fJ_p$, restricted to the subspace 
$\ker (R_{\lambda_k} I)$ in the space $L^2(\rho^\sharp)$; the unitary operator $\bvf$ is just the 
multiplication by $\xi$ in this representation. 

Thus, knowing inner function $\theta_s$, $s=\lambda_k$ we can construct the measure $\rho_k = 
\rho_s^\sharp=\rho_{\lambda_k}^\sharp$, and the   representation of the abstract spectral data 
(restricted to 
$\ker (R-\lambda_kI)$) in the space $L^2(\rho_k)$. Vice versa, knowing the measure $\rho_k$ we can 
construct the inner function $\theta_{s}$, $s=\lambda_k$. 

The case $s=\mu_k$ is similar, and even a bit simpler. We define the measure 
$\rho^{1}=\rho^{1}_s$ to be the Clark measure of the inner function 
$\theta^{1}_s$, and the measure $\rho^{1}_k$ being the reflection of $\rho^{1}_{s}$, 
$\rho^{1}_k = \rho^{1}_{s} \circ\fc$. The unitary operator $U=U_{s}: L^2(\rho^{1}_{s}) \to 
E_{1}(s) = \ker (|\Gamma_1| - s I)$ is defined as 
\begin{align*}
U f = \|u_{s}^{1}\|^{-1} u^{1}_{s} \cC^* f, \qquad f\in L^2(\rho^{1}_s), 
\end{align*}
where $\cC=\cC_s: \cK_{\theta^{1}_s}\to L^2(\rho^{1}_s)$ is the corresponding Clark operator. 
The operator $\cU=\cU_k : L^2(\rho^{1}_k) \to \ker (|\Gamma_1|-sI)$ is given by 
\begin{align*}
\left(\cU^* f\right)(\xi) = \left(U^* f\right)(\bar\xi)  =  \,\bigl[\cC  
(\|u^{1}_s\| f/u^{1}_s)\bigr](\bar \xi), \qquad 
f\in u^{1}_s \cK_{\theta^{1}_{s}}. 
\end{align*}
One can check that 
\begin{align*}
\left(\cU^* \fC\Gamma_1 \big|_{E(s) } \cU f\right)(\xi) = \xi \widebar {f(\xi)}  
\end{align*}
and that 
\begin{align*}
\|u^{1}_s\|^{-1}\cU^* u^{1}_s =\1. 
\end{align*}
so again, from the inner function $\theta^{1}_s$, $s=\lambda_k$ we can construct the measure 
$\rho^{1}_k$ (and all  relevant objects) and vice versa. 

\section{Appendix: Proof of the Abstract Borg's Theorem}
In this appendix, we will present the proof of the  Abstract Borg's Theorem (Theorem \ref{t:Borg 
01}). The proof follows \cite{Treil Liang 2022} and presented for the convenience of the reader.

First of all notice that the intertwining condition \eqref{e:Intertwine lambda mu} implies that the 
vector $p$ must is cyclic for $W$. Since everything is defined up to unitary equivalence, we can 
assume without loss of generality that $W$ is the multiplication $M_s$ by the independent variable 
$s$ in the weighted space $L^2(\rho)$, where $\rho$ is the spectral measure, corresponding to the 
vector $p$. 

Here the spectral measure can be defined as the unique (finite,  Borel, compactly 
supported) measure on $\R$ such that 
\begin{align*}
((W- zI)^{-1}p, p) = \int_\R \frac{d\rho(s)}{s-z}\qquad \forall z\in \C\setminus \R. 
\end{align*}
Since $W$ is a compact operator with eigenvalues $(\lambda_k^2)_{k\ge1}$, the measure $\rho$ is purely atomic, equivalently saying
\begin{align}
\label{e: spectral measure rho}
\rho = \sum_{k\ge 1} a_k \delta_{\lambda_k^2}
\end{align}
for some $a_k >0$. Note that in this representation the vector $p$ is represented by the function $1$ in $L^2(\rho)$. 

Moreover, the choice of the weights $a_k$ completely defines the triple $(W, W_1, p)$ (up to 
unitary equivalence). Indeed if the weights $a_k$ are known 
(recall that in our representation  $p\equiv1$), 
the operator $W_1$ is uniquely defined in the above spectral representation of $W$ as  $W_1 = W 
-pp^*$.  
\subsection{Preliminaries}
\label{s: C rho}
Let us recall some standard facts in perturbation theory. Let $W=W^*$ be a bounded self-adjoint operator, and let $p$ be its cyclic vector.  The spectral measure $\rho$ 
corresponding to the vector $p$ is defined as the unique Borel measure on $R$ such that 
\begin{align}
\label{e: spectral measure 02}
F(z):= ((W- zI)^{-1}p, p) = \int_\R \frac{d\rho(s)}{s-z}\qquad \forall z\in \C\setminus \R. 
\end{align}
As it is customary in perturbation theory, define a family of rank one perturbations
\begin{align*}
W^{\{\alpha\}} := W+ \alpha p p^*, \qquad \alpha\in\R . 
\end{align*}
Note that in this notation $W_1 = W^{\{-1\}} $. 

The spectral measures $\rho^{\{\alpha\}}$  are defined as the unique Borel measures such that 
\begin{align}
\label{e: spectral measure 03}
F^{\{\alpha\}}(z):= ((W^{\{\alpha\}}- zI)^{-1}p, p) = \int_\R 
\frac{d\rho^{\{\alpha\}}(s)}{s-z}\qquad \forall z\in \C\setminus \R; 
\end{align}
i.e.~the functions $F^{\{\alpha\}}$ are the Cauchy transforms of the measures $\rho^{\{\alpha\}}$. 

The relation between the Cauchy transforms $F$ and $F^{\{\alpha\}}$ 
is given by the famous Aronszajn--Krein formula,
\begin{align}
\label{e: F F^gamma}
F^{\{\alpha\}}= \frac{F}{1+\alpha F} , 
\end{align}
which is an easy corollary of the standard resolvent identities. 

Recall that our operator $W_1$ is exactly the operator $W^{\{-1\}}$, so let us for the consistency 
denote $F_1:= F^{\{-1\}}$. Identity \eqref{e: F F^gamma} can then be rewritten as 
\begin{align}
\label{e:F_1 via F}
F_1 & = \frac{F}{1-F}  
\intertext{so}
\label{e: F/F_1}
\frac{F}{F_1} &= 1-F. 
\end{align}

Hence to prove the first part of Theorem \ref{t:Borg 01} (existence and uniqueness of the triple $W$, 
$W_1$, $p$), it is sufficient to show that there exists a unique measure $\rho$ of form \eqref{e: 
spectral measure rho}
such that 
if $F$ is the Cauchy transform of $\rho$
\begin{align}
\label{e: C rho}
F(z) = \int_\R \frac{d\rho(s)}{s-z}\qquad \forall z\in \C\setminus \R
\end{align}
then the function $F_1$ defined by \eqref{e:F_1 via F} has the poles exactly at points $\mu_k^2$. 
Here the points $\lambda_k^2$ and $\mu_k^2$ are given, and the weights $a_k>0$ are to be found. 

\begin{rem}
\label{r:intertwine}
The intertwining relations \eqref{e:Intertwine lambda mu} follow immediately from \eqref{e:F_1 via 
F}. Indeed, 
one can see from \eqref{e:F_1 via F} that poles $\mu_k^2$ of $F_1$ are exactly the points $x\in\R$ 
where $F(x) =1$. The integral representation \eqref{e: C rho} can be written as 
\begin{align*}
F(x) = \sum_{k\ge1} \frac{a_k}{\lambda_k^2-x}.  
\end{align*}
Each term is positive for $x<\lambda_k^2$ and negative for $x>\lambda_k^2$, so it is not hard to 
conclude that
\begin{align*}
\lambda_1^2>\mu_1^2 >\lambda_2^2>\mu_2^2 \ldots> \lambda_k^2 > \mu_k^2  > \ldots\ge 0,
\end{align*}
which is equivalent to the intertwining relations \eqref{e:Intertwine lambda mu}. 
\end{rem}

\subsection{Guessing 
the function \texorpdfstring{$F$}{F}}
We want to reconstruct the spectral 
measure $\rho$ from the sequences $\{\lambda_k\}_{k\ge1}$ and $\{\mu_k\}_{k\ge1}$. 

Denote  
\begin{align*}
\sigma:= \sigma(W) = \{\lambda_k^2 \}_{k\ge1} \cup\{0\}, \qquad
\sigma_1:= \sigma(W_1) = \{\mu_k^2 \}_{k\ge1} \cup\{0\}. 
\end{align*}
It trivial that $F$ and $F_1$ are analytic in $\C\setminus \sigma$ and $\C\setminus \sigma_1$ 
respectively, and having simple poles at points $\lambda_k$ and $\mu_k$, $k\ge 1$ respectively 
(note that while the measure $\rho_1= \rho^{\{-1\}}$ can have a mass at $0$, the point $0$ is not 
an isolated singularity). 

Identity \eqref{e: F/F_1}  (which holds on $\C\setminus(\sigma\cup\sigma_1)$) implies that $F/F_1$ 
has simple poles at the points $\lambda_k^2$, $k\ge 1$, and that it can be analytically extended to 
$\C\setminus\sigma$. The (isolated) zeroes of $1-F$ must be at points where $F(z)=0$, so they must 
be only at the poles of $F_1$, i.e.~at the points $\mu_k^2$, $k\ge1$. 

So the function $F/F_1=1-F$ must be a function analytic on $\C\setminus \sigma$ with simple poles 
at the points $\lambda_k^2$, $k\ge 1$ and simple zeroes at the points $\mu_k^2$. 

One can try to guess such a function.  The simplest guess would be
\begin{align*}
1-F(z) = \prod_{k\ge 1 } \left( \frac{z- \mu_k^2  }{z-\lambda_k^2 } \right) \,.
\end{align*}

We will show that this is indeed the case, namely that the product converges and $1-F$ is indeed  
given by the above formula. 
From there we will trivially get the formula for $F$, 
and then compute the weights $a_k$.
\subsubsection{Convergence of the product
}
Denote 

\begin{align}
\label{e:Phi}
\Phi(z):= \prod_{k\ge 0 } \left( \frac{z- \mu_k^2  }{z-\lambda_k^2 } \right) \,.
\end{align}
We have guessed that $1-F(z) = \Phi(z)$. But first we need to show that the product \eqref{e:Phi}
is well defined. 

\begin{lm}
\label{l:Phi converges}
The product \eqref{e:Phi} converges uniformly on compact subsets of $\C\setminus \sigma$, and 
moreover
\begin{align}
\label{e: Phi(infty)=1}
\Phi(\infty) := \lim_{z\to \infty} \Phi(z) = 1
\end{align}
\end{lm}
\begin{proof}
Fix a compact $K\subset \C\setminus\sigma$
Trivially 
\begin{align}
\label{e: 1 - Phi_k}
\left| 1-\frac{z- \mu_k^2  }{z-\lambda_k^2 } \right| = \left| \frac{\lambda_k^2 - \mu_k^2  
}{z-\lambda_k^2 } \right|
\le C(K) (\lambda_k^2 - \mu_k^2) \qquad \forall z\in K . 
\end{align}
Since
\begin{align}
\label{e: sum la_k - mu_k}
\sum_{k\ge1} (\lambda_k^2 - \mu_k^2) \le \sum_{k\ge1} (\lambda_k^2 - \lambda_{k+1}^2) \le 
\lambda_1^2 <\infty
\end{align}
we have 
\begin{align*}
\sum_{k\ge 1} \left| 1-\frac{z- \mu_k^2  }{z-\lambda_k^2 } \right| \le C(K) \lambda_1^2 <\infty 
\qquad \forall z\in\C\setminus\sigma. 
\end{align*}
But this implies that the product \eqref{e:Phi} converges uniformly on compact subsets of 
$\C\setminus \sigma$. 

To prove the second statement, take any $R>2\lambda_1^2$. Then,  see \eqref{e: 1 - Phi_k} 
\begin{align*}
\left| 1-\frac{z- \mu_k^2  }{z-\lambda_k^2 } \right| = \left| \frac{\lambda_k^2 - \mu_k^2  
}{z-\lambda_k^2 } \right|
\le \frac2R (\lambda_k^2 - \mu_k^2) \qquad \forall z:\ |z|>R. 
\end{align*}
Using \eqref{e: sum la_k - mu_k} we get that for all $|z|>R$
\begin{align*}
\sum_{k\ge 1} \left| 1-\frac{z- \mu_k^2  }{z-\lambda_k^2 } \right| \le \frac{2\lambda_1^2}{R} , 
\end{align*}
which immediately implies \eqref{e: Phi(infty)=1}. 
\end{proof}

\subsection{Herglotz functions and the uniqueness}
To prove the uniqueness part of Theorem \ref{t:Borg 01}, it is sufficient to show  that our  
function $F$ must be 
exactly $1-\Phi(z)$. Let us recall the following definition.
\begin{defin}
An analytic function $f$ defined on the upper half-place is Herglotz (or 
Nevanlinna) if $\im f(z)\ge 0$ for all $z\in\C_+$.  
\end{defin}

Note, that the  Cauchy Transform of a measure supported on a real line is a Herglotz function. 
\begin{lm}
\label{l: properties Phi}
The function $\Phi$ defined by \eqref{e:Phi} satisfies the following properties
\begin{enumerate}
\item $\Phi(\overline z) = \overline{\Phi(z)}$; in particular, $\Phi(x)$ is real for all 
$x\in\R\setminus\sigma$. 

\item The function $\Phi$ has simple poles at points $\lambda_k^2$, and its only zeroes are the 
simple zeroes at points $\mu_k^2$. 

\item $\displaystyle \lim_{z\to \infty} \Phi(z) = 1$. 

\item The function $-\Phi$ is a Herglotz \tup{(}Nevanlina\tup{)}
 function in $\C_+$, i.e.~$\im \Phi(z) <0$ for 
all $z\in\C_+$.  
\end{enumerate}
\end{lm}

\begin{proof}[Proof of Lemma \ref{l: properties Phi}]
Statements \cond1, \cond2  are trivial, statement \cond3 was stated and proved in Lemma \ref{l:Phi 
converges} above. 

Let us prove statement \cond4. Let $\Arg$ denote the \emph{principal value} of the argument, taking 
values in the interval $(-\pi, \pi]$. It is easy to see that for $z\in \C_+$
\begin{align*}
  - \Arg \left( \frac{z- \mu_k^2  }{z-\lambda_k^2 } \right)
= - \Arg \left( \frac{\mu_k^2 -z }{\lambda_k^2 - z} \right) =\alpha_k(z)
\end{align*}
where $\alpha_k(z) > 0$ is the aperture of the angle at which an observer at the point $z\in\C_+$ 
sees the interval $[\mu_k^2, \lambda_k^2]$. Since for the whole real line $\R$ the aperture is 
$\pi$ and the intervals $[\mu_k^2, \lambda_k^2]$ do not intersecr $\R_-=(-\infty, 0)$,  we can 
conclude that
\begin{align*}
0 < -\sum_{k\ge1} \Arg \left( \frac{\mu_k^2 -z }{\lambda_k^2 - z} \right) < \pi , 
\end{align*}
so $-\im \Phi(z)>0$ for all $z\in\C_+$. 
\end{proof}

it is easy to see that the function $1-F$ satisfies  properties \cond1--\cond4 from  Lemma 
\ref{l: properties Phi}. So the fact that $F=1-\Phi$ follows from the Lemma below. 
\begin{lm}
\label{l:Phi unique} The function $\Phi$ defined by \eqref{e:Phi} is the only analytic in 
$\C\setminus\sigma$ function satisfying properties \cond1--\cond4 of Lemma \ref{l: properties Phi}. 
\end{lm}

\begin{proof}
Let $\Phi_1$ be another such function. 
Both functions have simple poles at $\lambda_k^2$ and simple zeros at $\mu_k^2$ (and these are the 
only isolated singularities and zeroes for both functions),  hence $\Psi(z):=\Phi(z)/\Phi_1(z)$ is 
analytic and zero-free in $\C\setminus \{0\}$.

Moreover, for $x\in \R\setminus\{0\}$ we have $\Psi(x)>0$. Indeed, on $\R\setminus 
\sigma\setminus\sigma_1$ the functions $\Phi_{1,2}$ are real and have the same sign, so $\Psi(x)>0$ 
on $\R\setminus \sigma\setminus\sigma_1$. Since $\Psi$ is continuous and zero-free on $\R \setminus 
\{0\}$, this tells us that $\Psi$ is positive on $\R \setminus \{0\}$. 

Next, let us notice that $\Psi(z)$ does not take real negative values. If $\im z>0$, then 
$\im\Phi(z)>0$, $\im\Phi_{1}(z)>0$, so $\Psi(z)=\Phi(z)/\Phi_1(z)$ cannot be negative real. If $\im 
z<0$, the symmetry $\Psi(\overline{z})= \overline{\Psi(z)}$ implies the same conclusion. And, as we 
just discussed above,  on the real line $\Psi$ takes positive real values. 

So, $\Psi$ omits infinitely many points, therefore by the Picard's Theorem the point $0$ is not an 
essential singularity for $\Psi$. Trivial analysis shows that $0$ cannot be a pole, otherwise 
$1/\Psi$ is analytic at 0, which also contradicts to the fact that $\Psi$ can't take negative real 
values. Hence the point $0$ is a removable singularity for function $\Psi$, so $\Psi$ is an entire 
function. 
By Liouville's Theorem, condition $\Psi(\infty) = 1$ implies that $\Psi\equiv 1$ for all $z \in 
\C$, so $\Psi\equiv\Psi_1$

We can avoid using Picard's theorem by considering the square root $\Psi^{1/2}$, where we take the 
principal branch of the square root (cut along the negative half-axis). Since $\Psi$ does not take 
negative real values, the function $\Psi^{1/2}$ is well defined (and analytic) on $\C\setminus 
\{0\}$.  Trivially $\re\Psi(z)^{1/2}\ge 0$, so by the Casorati--Weierstrass Theorem $0$ cannot be 
the essential singularity for $\Psi^{1/2}$. Again, trivial reasoning shows that $0$ cannot be a 
pole, so again, $\Psi^{1/2}$ is an entire function. The condition $\Psi^{1/2}(\infty)=1 $ then 
implies that $\Psi^{1/2}(z)\equiv 1$.
\end{proof}

\subsection{Existence and computing 
the spectral measure \texorpdfstring{$\rho$}{rho}}
%

To prove the existence part of Theorem \ref{t:Borg 01}, we just need to show that the function 
$F=1-\Phi$ is the Cauchy Transform of a measure $\rho$
\begin{align}
\label{e:rho 01}
\rho=\sum_{n\ge 1} a_n \delta_{\lambda_n^2} , \qquad a_n>0. 
\end{align}

One can almost get this result for free. 
Namely, the function $F=1-\Phi$ is a Herglotz function, analytic in $\C\setminus\sigma$, and 
statement 
\cond3 of Lemma \ref{l: properties Phi} implies that $F(\infty)=0$. Thus, 
it follows from general theory of Herglotz functions that the function that $F$ is a Cauchy 
transform of a measure $\rho$ supported on $\sigma$. Moreover, since $F$ has simple poles at points 
$\lambda_n^2$, one can see that $\rho(\{\lambda_n^2\})\ne 0$.  Therefore, to prove existence, it 
remains  to show that the measure $\rho$ does not have an atom at $0$. 

But a proof of this fact  
would require some computation, which is not much easier than 
an elementary  proof using the partial fraction decomposition of the function 
$\Phi$ presented below.   This proof does not use any theory of Herglotz functions, and also gives 
formulae for weights $a_n$.  
\begin{lm}
\label{l: Phi decomposition}
The function $\Phi$ defined by \eqref{e:Phi} can be decomposed as

\begin{align}
\label{e: Phi 03}
\Phi(z) = 1 -  \sum_{n\ge 1} \frac{a_n}{\lambda_n^2-z} , 
\end{align}
where 

\begin{align}
\label{e: a_n}
a_n = (\lambda_n^2 - \mu_n^2 ) \prod_{k\ne n} \left(  \frac{\lambda_n^2 - \mu_k^2}{\lambda_n^2 - 
\lambda_k^2}     \right)
\end{align}
\end{lm}

This representation immediately implies that $F$ is the Cauchy Transform of the measure $\rho$ 
given by \eqref{e:rho 01} with weights $a_n$ defined by \eqref{e: a_n}. 

\begin{proof}[Proof of Lemma \ref{l: Phi decomposition}]
Consider functions $\Phi\ci N$
\begin{align*}
\Phi\ci N(z) = \prod_{k=1}^N \left( \frac{z- \mu_k^2  }{z-\lambda_k^2 } \right)
\end{align*}
Trivially
\begin{align}
\label{e: Phi_N}
\Phi\ci N(z) = 1 -  \sum_{n\ge 1} \frac{a_n^N}{\lambda_n^2-z} 
\end{align}
where 

\begin{align*}
a_n^N = (\lambda_n^2-\mu_n^2) \prod_{\substack{k=1 \\ k\ne n}}^N \left(  \frac{\lambda_n^2 - 
\mu_k^2}{\lambda_n^2 - \lambda_k^2}     \right)  \qquad \text{if } n\le N, 
\end{align*}
and $a_n^N=0$ if $n>N$. 

 We know, see Lemma \ref{l:Phi converges} that $\Phi_{N}(z)$ converges   to $\Phi(z)$ uniformly on 
 compact subsets of $  \C \setminus \sigma$.

 Hence, to prove the lemma it remains to show that 

\begin{align*}
\sum\limits_{n= 1}^N \frac{a_n^N}{\lambda_n^2-z} \to \sum\limits_{n\ge 1} \frac{a_n}{\lambda_n^2-z} 
\qquad \text{as }N\to\infty
\end{align*}
uniformly on compact subsets of $\C\setminus \sigma$. 

Take $z = 0$ in \eqref{e: Phi_N}. Then
\begin{align*}
1-\sum\limits_{n \geq 1}\frac{a_n^N}{\lambda_{n}^2}=\prod\limits_{k=1}^{N}\left( \frac{ \mu_k^2  
}{\lambda_k^2 } \right) >0, 
\end{align*}
so $\sum\limits_{n \geq 1}\frac{a_n^N}{\lambda_{n}^2} \le 1$.

Notice that for any fixed $n$ the sequence $a_n^N \nearrow a_n$  as $N\to\infty$, so 
$\sum\limits_{n \geq 1}\frac{a_n}{\lambda_n^2} \le 1$. 

Take an arbitrary compact $K\subset \C\setminus \sigma$. Clearly for any $z\in K$
\[
\left| \frac{a_n^N}{\lambda_n^2 - z} \right| \le \frac{a_n^N}{\dist(K, \sigma)} \le 
\frac{a_n}{\dist(K, \sigma)} 
\le \frac{\lambda_1^2}{\dist(K, \sigma)} \cdot \frac{a_n}{\lambda_n^2}, 
\]
so the condition $\sum_{n\ge 1} a_n/\lambda_n^2 \le 1$ implies that 
the series $\sum\limits_{n\ge 1} \frac{a_n^N}{\lambda_n^2-z}$ converges uniformly on the compact  
$K$. 
\end{proof}
\subsection{
The trivial kernel condition}

To complete the proof of Theorem \ref{t:Borg 01}, it remains to show that the condition $\|W^{-1/2} 
p\|=1$ is equivalent to \eqref{e: norm q = 1 02}, and that if \eqref{e: norm q = 1 02} holds, then 
the condition $\|W^{-1} p\|=\infty$ is equivalent to \eqref{e: q not in ran R 02}. 

Let us first investigate the condition $\|W^{-1/2} p\|=1$. The operator $W$ is the multiplication 
by the independent variable $s$ in the weighted space $L^2(\rho)$, where $\rho= \sum_{k\ge1} a_n 
\delta_{\lambda_n^2}$ with $a_n$ given by 
\eqref{e: a_n}. 
Recall, see Lemma \ref{l: Phi decomposition}, that 
\begin{align}
\label{e: Phi decomp 01}
\prod\limits_{k=1}^{\infty}\Bigg( \frac{z-\mu_{k}^{2}}{z-\lambda_{k}^{2}} \Bigg)
= 1 - \sum\limits_{k=1}^{\infty}\frac{a_k}{\lambda_{k}^{2}-z}
\end{align}  
Plugging the real $x<0$ into \eqref{e: Phi decomp 01} and 
taking the limit of both sides as $x\to 0^-$ we get that   
\begin{align*}
\prod\limits_{k=1}^{\infty}\frac{\mu_{k}^{2}}{\lambda_{k}^{2}}  = 
1-\sum\limits_{k=1}^{\infty}\frac{a_{k}}{\lambda_k^2};  
\end{align*}
the interchange of limit and sum (product) can be justified, for example, by the monotone 
convergence theorem. 

Therefore $\sum_{k\ge 1} a_k/\lambda_k^2 = 1$ if and only if $\prod_{k\ge0} 
(\mu_k^2/\lambda_k^2)=0$. The latter condition is equivalent to 
\begin{align*}
\sum_{k\ge1} \left( 1- \frac{\mu_k^2}{\lambda_k^2}  \right) =\infty,
\end{align*}
which  is exactly the condition \eqref{e: norm q = 1 02}.

Now, let us investigate the condition $\|W^{-1} p\|=\infty$. It can be rewritten as 
\begin{align}
\label{e: sum a lambda^4}
\sum_{k\ge 1} \frac{a_k}{\lambda_k^4} = \infty.
\end{align}

Rewriting identity \eqref{e: Phi decomp 01} as
\begin{align*}
\prod_{k=1}^{\infty} \left( \frac{z- \mu_k^2  }{z-\lambda_k^2 } \right) &= 
1-\sum\limits_{k=1}^{\infty}a_{k}\frac{(\lambda_{k}^2-z)+z}{\lambda_{k}^2(\lambda_{k}^2-z)} \\
&= 1-\sum\limits_{k \geq 1}\frac{a_k}{\lambda_k^2} - \sum\limits_{k \geq 1}\frac{a_k z}{\lambda_k^2 
(\lambda_k^2 -z)} ,
\end{align*}
we see that if the condition \eqref{e: norm q = 1 02} holds, then 
\begin{align*}
-\frac{1}{z}\prod\limits_{k=1}^{\infty}\Bigg( \frac{z-\mu_{k}^{2}}{z-\lambda_{k}^{2}} 
\Bigg)=\sum\limits_{k=1}^{\infty}\frac{a_k}{\lambda_k^2(\lambda_{k}^{2}-z)} .  
\end{align*}

Substituting $z=-\lambda_N^2$ we get  that 
\begin{align*}
\frac{1}{\lambda_N^2}\prod\limits_{k=1}^{\infty}\Bigg( \frac{\lambda_N^2 + \mu_{k}^{2}}{\lambda_N^2 
+\lambda_{k}^{2}} \Bigg)=\sum\limits_{k=1}^{\infty}\frac{a_k}{\lambda_k^2(\lambda_N^2 
+\lambda_{k}^{2})} .
\end{align*}
By the monotone convergence theorem 
\begin{align*}
\lim_{N\to\infty} \sum\limits_{k\ge 1}\frac{a_k}{\lambda_k^2(\lambda_N^2 +\lambda_{k}^{2})} =
\sum\limits_{k\ge 1}\frac{a_k}{\lambda_k^4} , 
\end{align*}
so the condition \eqref{e: sum a lambda^4} can be rewritten as 
\begin{align}
\label{e: prod1 = infty}
\lim_{N\to\infty} \frac{1}{\lambda_N^2}\prod\limits_{k=1}^{\infty}
\Bigg( \frac{\lambda_N^2 + \mu_{k}^{2}}{\lambda_N^2 +\lambda_{k}^{2}} \Bigg) =\infty
\end{align}
The following simple lemma completes the proof. 

\begin{lm}
\label{l: prod = infy 02}
The condition \eqref{e: prod1 = infty} holds if and only if 
\begin{align*}
\sum_{k\ge 1}\left(\frac{\mu_k^2}{\lambda_{k+1}^2} -1\right) = \infty .
\end{align*}
\end{lm}

\begin{proof}
First of all notice that 
\begin{align}
\label{e: prod bounds 01}
0 < C \leq \prod\limits_{k=N}^{\infty}\Bigg( 
\frac{\lambda_{N}^{2}+\mu_{k}^{2}}{\lambda_{N}^{2}+\lambda_{k}^{2}} \Bigg) \leq 1 
\end{align}
with $C>0$ independent of $N$. 

Indeed, for all $k\ge N$ we trivially have
\begin{align}
\label{e: term bound 01}
\frac12 \le \frac{\lambda_{N}^{2}+\mu_{k}^{2}}{\lambda_{N}^{2}+\lambda_{k}^{2}} \le 1 .
\end{align}
The upper bound in \eqref{e: term bound 01} trivially implies the upper bound in \eqref{e: prod 
bounds 01}. 

To get the lower bound in \eqref{e: prod bounds 01} we use the estimate \eqref{e: term bound 01} 
and the inequality
\begin{align*}
\ln x \ge (\ln2) (x-1), \qquad \forall\  x \in [1/2, 1]. 
\end{align*}
Thus 
\begin{align*}
\sum_{k=N}^\infty \ln2 \left( \frac{\lambda_{N}^{2}+\mu_{k}^{2}}{\lambda_{N}^{2}+\lambda_{k}^{2}} 
-1   \right) 
= \sum_{k=N}^\infty  -\ln2  \frac{\lambda_{k}^{2} - \mu_{k}^{2}}{\lambda_{N}^{2}+\lambda_{k}^{2}}
\ge -\ln2 \sum_{k=N}^{\infty} \frac{\lambda_k^2-\mu_k^2}{\lambda_N^2}
\ge -\ln2
\end{align*}
we see that the lower bound in \eqref{e: prod bounds 01} holds with $C_1=\frac{1}{2}$. 

The estimate \eqref{e: prod bounds 01} implies that the condition \eqref{e: prod1 = infty} is 
equivalent to 
\begin{align*}
\lim_{N\to\infty} \frac{1}{\lambda_N^2}\prod\limits_{k=1}^{N-1}
\Bigg( \frac{\lambda_N^2 + \mu_{k}^{2}}{\lambda_N^2 +\lambda_{k}^{2}} \Bigg) =\infty.  
\end{align*}
Since $\lambda_N^2 \le \lambda_N^2 + \lambda_k^2 \le 2 \lambda_N^2$ for $k\ge N$, the above 
condition is equivalent to 
\begin{align}
\label{e: prod1 = infty 02}
\lim_{N\to\infty} 
\prod\limits_{k=1}^{N-1}
\Bigg( \frac{\lambda_N^2 + \mu_{k}^{2}}{\lambda_N^2 +\lambda_{k+1}^{2}} \Bigg) =\infty, 
\end{align}
Denote
\begin{align*}
G_{N}(z) := 
\prod\limits_{k=1}^{N-1}\Bigg( \frac{\mu_{k}^{2} - z}{\lambda_{k+1}^{2} - z} \Bigg)
\end{align*}
The function $G_N$ is analytic in the half-plane $\re z < \lambda_{N}^2$ and satisfies the 
inequality $|G_N(z)|\ge 1$ there. Therefore the function $\ln|G_N|$ is harmonic and non-negative in 
the disc $D_N$ of radius $2\lambda_N^2$ centered at $-\lambda_N^2$.  So by the Harnack inequality 
\begin{align*}
\frac13 \ln |G_N(-\lambda_N^2)| \le \ln |G_N(0)| \le 3  \ln |G_N(-\lambda_N^2)| .
\end{align*}

Note that $G_N(0) = \prod_{k=1}^{N-1} \mu_k^2/\lambda_{k+1}^2 $, so the condition \eqref{e: prod1 = 
infty 02} (which translates to the condition $\lim_{N\to\infty} G\ci N (-\lambda_N^2) =\infty$)  is 
equivalent to 
\begin{align*}
\lim_{N\to\infty} G\ci N(0) =\lim_{N\to\infty} \prod_{k=1}^{N-1} \frac{\mu_k^2}{\lambda_{k+1}^2} 
=\infty.
\end{align*}
The latter condition is equivalent to the condition 
\begin{align*}
\prod_{k\ge 1} \frac{\mu_k^2}{\lambda_{k+1}^2} = \infty, 
\end{align*}
which in turn is equivalent to \eqref{e: q not in ran R 02}. 
\end{proof}

\def\cprime{$'$}
  \def\lfhook#1{\setbox0=\hbox{#1}{\ooalign{\hidewidth\lower1.5ex\hbox{'}\hidewidth\crcr\unhbox0}}}
\providecommand{\bysame}{\leavevmode\hbox to3em{\hrulefill}\thinspace}
\providecommand{\MR}{\relax\ifhmode\unskip\space\fi MR }
\providecommand{\MRhref}[2]{%
  \href{http://www.ams.org/mathscinet-getitem?mr=#1}{#2}
}

\end{document}